\documentclass{article}
\usepackage{amsmath}
\usepackage{amssymb, amscd, amsfonts, latexsym, amsthm}
\usepackage{enumerate}
\usepackage[dvips]{graphicx,color}

\newcommand{\no}{\noindent}
\newtheorem{thm}{Theorem}[section]
\newtheorem{prop}[thm]{Proposition}
\newtheorem{cor}[thm]{Corollary}
\newtheorem{rem}[thm]{Remark}
\newtheorem{definition}[thm]{Definition}
\newtheorem{ex}[thm]{Example}
\newtheorem{lem}[thm]{Lemma}

\numberwithin{equation}{section}

\newcommand{\R}{\mathbb{R}}

\newcommand{\ds}{\displaystyle}

\newcommand{\sm}{\setminus}
\newcommand{\pd}{\partial}
\newcommand{\al}{\alpha}

\newcommand{\de}{\delta}
\newcommand{\De}{\Delta}
\newcommand{\ep}{\varepsilon}
\newcommand{\la}{\lambda}

\newcommand{\ome}{\omega}
\newcommand{\Ome}{\Omega}

\renewcommand{\(}{\left(}
\renewcommand{\)}{\right)}
\renewcommand{\lvert}{\left\vert}
\renewcommand{\rvert}{\right\vert}

\DeclareMathOperator{\diam}{diam}
\DeclareMathOperator{\conv}{conv}

\DeclareMathOperator{\Vol}{Vol}

\DeclareMathOperator{\Refl}{Refl}
\DeclareMathOperator{\Span}{Span}

\DeclareMathOperator{\slope}{slope}
\DeclareMathOperator{\sign}{sign}
\DeclareMathOperator{\Ker}{Ker}

\def\c#1{\overset{\mbox{\tiny $\circ$}}{#1}} 

\begin{document}

\title{\bf Experimental investigation on the uniqueness of a center of a body}

\author{Shigehiro Sakata}

\date{\today}

\maketitle

\begin{abstract} 
The object of our investigation is a point that gives the maximum value of a potential with a strictly decreasing radially symmetric kernel. It defines a center of a body in $\R^m$. When we choose the Riesz kernel or the Poisson kernel as the kernel, such centers are called an $r^{\al -m}$-center or an illuminating center, respectively.

The existence of a center is easily shown but the uniqueness does not always hold. Sufficient conditions of the uniqueness of a center have been studied by some researchers. The main results in this paper are some new sufficient conditions for the uniqueness of a center of a body.\\   

\no{\it Keywords and phrases}. Radial center, $r^{\al -m }$-center, illuminating center, minimal unfolded region, heart, Alexandrov's reflection principle, moving plane method.\\
\no 2010 {\it Mathematics Subject Classification}: 52A40, 52A10, 52A20, 51M16, 35B38, 26B25.
\end{abstract}
\section{Introduction}
Let $\Ome$ be a body (the closure of a bounded open set) in $\R^m$. We consider a potential of the form
\begin{equation}
K_\Ome (x) = \int_\Ome k(r) dy,\ x \in \R^m ,\ r= \lvert x-y \rvert .
\end{equation} 
If the kernel $k:(0,+\infty) \to \R$ is strictly decreasing and satisfies the condition $(C^0_\al)$ (which is detailed in section 2), then the potential $K_\Ome$ is continuous on $\R^m$ (Proposition \ref{regularity}) and has a maximum point only in the convex hull of $\Ome$ (Proposition \ref{exist}). We call a maximum point of $K_\Ome$ a {\it $k$-center} of $\Ome$ in what follows. The object of our investigation in this paper is a $k$-center of $\Ome$.

Analytically, the study on $k$-centers is related to the investigation on the shape of a solution of a partial differential equation. When the kernel $k(r)$ is given by the Gauss kernel $(4\pi t)^{-m/2} \exp (-r^2 /(4t))$ with a positive parameter $t$, we obtain the unique bounded solution of the Cauchy problem for the heat equation with initial datum $\chi_\Ome$. A (spatial) maximum point of the solution of the heat equation is called a {\it hot spot}. The existence, asymptotic behavior, uniqueness and location of a hot spot were well-studied, for example, in \cite{BL, BMS, CK, JS, MS}. When the kernel $k(r)$ is given by the Poisson kernel $h(r^2 +h^2)^{-(m+1)/2}$ with a positive parameter $h$, we obtain the Poisson integral for the upper half-space (up to a constant multiple). The Poisson integral is a solution of the Laplace equation for the upper half-space. A maximum point of the Poisson integral was studied in \cite{Sak}. 

Geometrically, the study on $k$-centers is related to Moszy\'{n}ska's radial center. In \cite{M1}, she introduced a {\it radial center} of a star body $A$ induced by a function $\phi$ as a maximum point of the function 
\begin{equation}\label{Phi}
\Phi_A(x)= \int_{S^{m-1}} \phi \( \rho_{A-x} (v) \) d\sigma (v),\ x \in \Ker A .
\end{equation}
Here, $\rho_{A-x} (v) = \max \{ \la \geq 0 \vert \la v + x \in A \}$ is the {\it radial function} of $A$ with respect to $x$, and $\Ker A = \{ p \in A \vert \forall q \in A,\ \overline{pq} \subset A\}$ is the {\it kernel} of $A$. Her motivation for the study on radial centers comes from the optimal position of the origin for the {\it intersection body} of a star body. Intersection bodies were introduced by Lutwak in \cite{L} to solve Busemann and Petty's problem \cite{BP}. We refer to Moszy\'{n}ska's text book \cite[pp. 185--201]{M2} for those historical backgrounds in convex geometry. The paper \cite{HMP} is also a good reference for the physical meaning of radial centers.

Using the polar coordinate, we rewrite the function $\Phi_A (x)$ as 
\begin{equation}
\Phi_A (x) = \int_A \phi' (r) r^{1-m} dy + \phi (0) \sigma \( S^{m-1} \) ,\ x \in \Ker A ,\ r=\lvert x-y \rvert .
\end{equation}
Putting $k(r) =\phi'(r) r^{1-m}$, we obtain the potential $K_A$. Since the potential $K_A$ is defined on $\R^m$ even if $A$ is NOT star-shaped, we can understand that the notion of $k$-centers is an extension of radial centers.

When the kernel $k(r)$ is given by the monomial $r^{\al -m}$, $k$-centers are well-studied. In \cite{M1}, when $\phi (\rho ) = \rho^{\al}$ in \eqref{Phi}, Moszy\'{n}ska called a maximum point of $\Phi_A$ a {\it radial center of order $\al$} and showed that if $0<\al \leq 1$, then every convex body has a unique radial center of order $\al$. In \cite{H1}, for $\al >1$, the uniqueness of a radial center of a convex body was studied but the argument included an error. In \cite{H2}, Herburt studied the location of a radial center of order 1. She showed that every smooth convex body has a radial center of order 1 only in its interior. In \cite{O1}, O'Hara investigated the potential
\begin{equation}
V_\Ome^{(\al )} (x) =
\begin{cases}
\ds \sign (m-\al ) \int_\Ome r^{\al -m} dy &(0< \al \neq m),\\
\ds -\int_\Ome \log r dy &(\al =m ),
\end{cases}
\ x \in \R^m ,\ r=\lvert x-y \rvert .
\end{equation}
He called the potential $V_\Ome^{(\al )}$ the {\it $r^{\al -m}$-potential} and defined an {\it $r^{\al -m}$-center} of $\Ome$ as a maximum point of $V_\Ome^{(\al )}$. In other words, he extended the notion in \cite{M1} to a non-star-shaped case. He showed that if $\al \geq m+1$, then every body has a unique $r^{\al -m}$-center.

On the uniqueness of a $k$-center in \cite{M1, O1}, the common idea is to show the strict concavity of the potential $K_\Ome$ on the convex hull of $\Ome$ (the location of $k$-centers). But, using {\it Alexandrov's reflection principle} or the {\it moving plane method} (\cite{GNN, Ser}), we can restrict a region containing all $k$-centers smaller than the convex hull of $\Ome$. We call such a small region the {\it minimal unfolded region} of $\Ome$, denoted by $Uf(\Ome )$, which was introduced by O'Hara in \cite{O1}. When $\Ome$ is a convex body, in \cite{BMS}, the minimal unfolded region was independently defined by Brasco, Magnanini and Salani as the {\it heart} of $\Ome$ denoted by $\heartsuit (\Ome )$. Hence, in order to show the uniqueness of a $k$-center, it is sufficient to show the strict concavity of $K_\Ome$ on the minimal unfolded region.

The minimal unfolded region $Uf(\Ome )$ is made by the following procedure: Fix a direction $v \in S^{m-1}$ and a parameter $b \in \R$; Let $\Refl_{v,b}$ denote the reflection of $\R^m$ in the hyperplane $\{ z \in \R^m \vert z \cdot v =b \}$; By $\Ome^+_{v,b} = \{ z \in \Ome \vert z \cdot v \geq b \}$, we denote the set of all points in $\Ome$ whose height in the direction $v$ are not smaller than $b$; We repeat to fold the set $\Ome^+_{v,b}$ by the reflection $\Refl_{v,b}$ and to gradually decrease the value $b \in \R$ until the image is stick out from $\Ome$; Let $l(v)$ be the {\it minimum folding height} for $v$, that is, put
\begin{equation}
l(v) = \min\left\{ a \in \R \lvert \forall b \geq a,\ \Refl_{v,b}\( \Ome^+_{v,b} \) \subset \Ome \right. \right\} ;
\end{equation}
Define the minimal unfolded region of $\Ome$ by
\begin{equation}\label{uf}
Uf(\Ome )= \bigcap_{v \in S^{m-1}} \left\{ \left. z \in \R^m \rvert z \cdot v \leq l(v) \right\} .
\end{equation}

For example, in $\R^2$, the minimal unfolded region of the union of the two same-sized discs
\begin{equation}\label{D2}
D_1 \cup D_2 = \left\{ \( y_1 ,y_2 \) \lvert \( y_1 +1 \)^2 +y_2^2 \leq 1 \right\} \right. \cup \left\{ \( y_1 ,y_2 \) \lvert \( y_1 -1 \)^2 +y_2^2 \leq 1 \right\} \right.
\end{equation}
is the line segment $\{ (y_1 ,0) \vert -1\leq y_1 \leq 1\}$. Therefore, when we investigate the number of $k$-centers of $D_1 \cup D_2$, we should consider the graph of the function $K_{D_1 \cup D_2}(\la ,0)$ for $-1 \leq \la \leq 1$. Then, for a given concrete kernel, we can draw the graph of $K_{D_1 \cup D_2}(\la ,0)$ with the calculator Maple. In such a manner, we give some examples of the graphs of $r^{\al -m}$-potentials. To be precise, we produce the following examples:
\begin{enumerate}
\item[(1)] The union of the two same-sized discs (\ref{D2}) has two $r^{-1/2}$-centers.
\item[(2)] The set of $r^{-1/2}$-centers of the annulus $\{ (y_1 ,y_2) \vert 1 \leq y_1^2 +y_2^2 \leq 4\}$ is a circle.
\item[(3)] The isosceles triangle $\{ (y_1 ,y_2) \vert 0\leq y_1 \leq 1,\ \vert y_2 \vert \leq ( \tan ( \pi /10 ) ) y_1 \}$ has a unique $r^{-1/2}$-center.
\item[(4)] The cone $\{ (y_1 ,y_2,y_3) \vert 0\leq y_1 \leq 1,\ y_2^2 +y_3^2 \leq (\tan^2 ( \pi /10 ) ) y_1^2 \}$ has a unique $r^{-1/2}$-center.
\item[(5)] The body of revolution $\{ (y_1 ,y_2,y_3) \vert 0\leq y_1 \leq 1,\ y_2^2 +y_3^2 \leq ( \tan^2 ( \pi /10 ) ) y_1 \}$ has a unique $r^{-1/2}$-center.
\end{enumerate}
From the third example, we see that, in general, the $r^{\al -m}$-potential is not concave on the convex hull of a body for $1<\al <m+1$. Hence it seems difficult to give a sufficient condition for the uniqueness of an $r^{\al -m}$-center for $1 < \al <m+1$.

Our main result in this paper is a sufficient condition for the uniqueness of a $k$-center implying the examples (3) and (5). Precisely, if the kernel $k$ satisfies the condition $(C^1_\al )$ for some $\al >1$ (which is detailed in section 2), and if $k'(r)/r$ is increasing, then the body of revolution
\begin{equation}\label{revolution_omega}
\Ome = \left\{ \( y_1 ,\bar{y} \) \in \R \times \R^{m-1} \lvert 0\leq y_1 \leq 1,\ \lvert \bar{y} \rvert \leq \ome \( y_1 \) \right\} \right. ,
\end{equation}
where $\ome :[0,1] \to [0,+\infty )$ is a piecewise $C^1$ function with $\ome^{m-1}$ concave, has a unique $k$-center. This result immediately implies the uniqueness of an $r^{\al -m}$-center of the body (\ref{revolution_omega}) for $1<\al <m+1$. Furthermore, using the same manner as in the above result, we also show that a non-obtuse triangle in $\R^2$ has a unique $k$-center if $k'(r)/r$ is increasing. We remark that these results on the uniqueness of an $r^{\al -m}$-center cannot follow from the power-concavity argument as in \cite{BL}.\\

Throughout this paper, $\conv X$, $\diam X$, $\c{X}$ (or $X^\circ$) and $X^c$ denote the convex hull, the diameter, the interior and the complement of a set $X$ in $\R^m$, respectively. We denote the spherical Lebesgue measure of any $N$-dimensional space by $\sigma_N$.\\

\no{\bf Acknowledgements.} The author would like to express his deep gratitude to his advisor Jun O'Hara for giving kind advice to him. 
\section{Preliminaries}
In this section, we introduce necessary results in \cite{BMS, BM, O1, Sak} for our study.

For an $\al >0$, we define the condition $(C^0_\al )$ of a function $k :(0,+\infty) \to \R$ as
\begin{enumerate}
\item[$( C^0_\al )$] $k$ is continuous on the interval $(0, +\infty )$, and 
\begin{equation}
k(r) = 
\begin{cases}
O\( r^{\al -m} \)  \ &( \al < m ),\\
O\( \log r \)      \ &( \al = m ),\\
O\(1 \)            \ &( \al > m ),
\end{cases} 
\end{equation} 
as $r$ tends to $0^+$.
\end{enumerate}
For an $\al >1$, we define the condition $(C^1_\al )$ of a function $k :(0,+\infty) \to \R$ as
\begin{enumerate}
\item[$(C^1_\al )$] $k$ is once continuously differentiable on the interval $(0, +\infty )$, and 
\begin{equation}
k(r) = 
\begin{cases}
O\( r^{\al -m} \)  \ &( \al < m ),\\
O\( \log r \)      \ &( \al = m ),\\
O\(1 \)            \ &( \al > m ),
\end{cases} \
k '(r) = 
\begin{cases}
O\( r^{\al -m -1} \)  \ &( \al < m+1 ),\\
O\( \log r \)         \ &( \al = m+1 ),\\
O\(1 \)               \ &( \al > m+1 ),
\end{cases}
\end{equation} 
as $r$ tends to $0^+$.
\end{enumerate}

Let $\Ome$ be a body (the closure of a bounded open set) in $\R^m$, and  
\begin{equation}\label{K}
K_\Ome (x) = \int_\Ome k(r) dy,\ x \in \R^m ,\ r=\lvert x-y \rvert .
\end{equation}
We always assume that the kernel $k$ satisfies any of the conditions $(C^0_\al)$ or $(C^1_\al)$. We denote a point $x$ in $\R^m$ by $x=(x_1,\ldots, x_m)$ and a point $y$ in $\Ome$ by $y=(y_1,\ldots,y_m)$. We understand that the letter $r$ is always used for $r=\vert x  - y \vert$. 
\subsection{Properties of {\boldmath $K_\Ome$}}
Let us prepare some properties of our potential $K_\Ome$ from \cite{Sak} without those proofs.

\begin{prop}[{\cite[Propositions 2.3, 2.6, 4.1 and Corollary 4.2]{Sak}}]\label{regularity}
Let $\Ome$ be a body in $\R^m$.
\begin{enumerate}[$(1)$]
\item If the kernel $k$ satisfies the condition $(C^0_\al)$ for some $\al >0$, then the potential $K_\Ome$ is continuous on $\R^m$.
\item If the kernel $k$ satisfies the condition $(C^1_\al )$ for some $\al >1$, then the potential $K_\Ome$ is of class $C^1$ on $\R^m$, and we have 
\[
\frac{\pd K_\Ome}{\pd x_j}(x) = \int_\Ome \frac{\pd}{\pd x_j}k(r) dy,\ x \in \R^m .
\]
\item If $\Ome$ has a piecewise $C^1$ boundary, and if the kernel $k$ satisfies the condition $(C^1_\al )$ for some $\al >1$, then the potential $K_\Ome$ is of class $C^2$ on $\R^m \sm \pd \Ome$, and we have
\begin{align*}
\frac{\pd K_\Ome}{\pd x_j}(x) &= -\int_{\pd \Ome} k(r) e_j \cdot n(y) d\sigma (y),\ x \in \R^m ,\\
\frac{\pd^2 K_\Ome}{\pd x_i \pd x_j}(x) &= -\int_{\pd \Ome} \frac{\pd}{\pd x_i}k(r) e_j \cdot n(y) d\sigma (y),\ x \in \R^m \sm \pd \Ome ,
\end{align*} 
where $n$ is the outer unit normal vector field of $\pd \Ome$, and $e_j$ is the $j$-th unit vector of $\R^m$.
\end{enumerate}
\end{prop}




\begin{prop}[{\cite[Proposition 3.2]{Sak}}]\label{exist}
Let $\Ome$ be a body in $\R^m$. Suppose that the kernel $k$ is strictly decreasing and satisfies the condition $(C_\al^0)$ for some $\al >0$. The potential $K_\Ome$ has a maximum point, and any maximizer of $K_\Ome$ belongs to the convex hull of $\Ome$.
\end{prop}


\begin{definition}[{\cite[Definition 3.3]{Sak}}]\label{center}
{\rm Let $\Ome$ be a body in $\R^m$. A point $x$ is called a {\it $k$-center} of $\Ome$ if it gives the maximum value of $K_\Ome$.}
\end{definition}

\subsection{Properties of minimal unfolded regions}
Let $Uf(\Ome )$ be the minimal unfolded region of a body $\Ome$ as in (\ref{uf}). We introduce some properties of the minimal unfolded region of $\Ome$ from \cite{BM, BMS, O1, Sak} with slight modifications in our case. (The studies performed in \cite{BM} does not ask for the regularity of $k$ but required the boundedness of $k(r)$ at $r=0^+$.) Geometric properties of the minimal unfolded region were also studied in \cite{O2}.

\begin{rem}[{\cite[p. 381]{O1}}]\label{rem_uf}
{\rm 
Let $\Ome$ be a body in $\R^m$.
\begin{enumerate}[(1)]
\item The centroid (the center of mass) of $\Ome$ is contained in $Uf(\Ome )$. Hence $Uf(\Ome )$ is not empty.
\item $Uf(\Ome )$ is contained in $\conv \Ome$ but, in general, not contained in $\Ome$.
\item $Uf(\Ome)$ is compact and convex.
\end{enumerate}
}
\end{rem}






\begin{ex}[{\cite[Lemma 5]{BM}, \cite[Example 3.4]{O1}}]\label{uf_triangle}
{\rm 
\begin{enumerate}
\item[(1)] The minimal unfolded region of a non-obtuse triangle is given by the polygon formed by the mid-perpendicular of edges and the bisectors of angles. In particular, it is contained in the triangle formed by joining the middle points of the edges.  
\item[(2)] The minimal unfolded region of an obtuse triangle is given by the polygon formed by the largest edge, its mid-perpendicular and the bisectors of angles.
\end{enumerate}
}
\end{ex}

\begin{prop}[{\cite[Proposition 4.9]{Sak}}]\label{exist_uf}
{\rm Let $\Ome$ be a body in $\R^m$. If $k$ is strictly decreasing, then any $k$-center of $\Ome$ belongs to $Uf(\Ome)$.}
\end{prop}




We give a relation between a body $\Ome$ and its minimal unfolded region $Uf(\Ome)$. The idea of the proof is due to \cite[Theorem 1]{BM}. To be precise, in \cite{BM}, Brasco and Magnanini studied geometry of the minimal unfolded region (heart) of a {\em convex} body, but their argument works for a {\em non-convex} body with slight modifications.

\begin{lem}\label{l_semicont}
Let $\Ome$ be a body in $\R^m$. The minimum folding height $l :S^{m-1} \to \R$ is lower semi-continuous.
\end{lem}

\begin{proof}
We show that the set $\{ v \in S^{m-1} \vert l(v) >b \}$ is open in $S^{m-1}$ for any $b \in \R$. Fix an arbitrary $b \in \R$. Let $w \in S^{m-1}$ be a direction with $l(w)>b$. 

We first show that the non-empty intersection $\Refl_{w,b}(\Ome^+_{w,b}) \cap \Ome^c$ has an interior point. We take a point $x$ from the intersection. Since $\Ome^c$ is open in $\R^m$, there exists an $\ep_1 >0$ such that the $\ep_1$-neighborhood of $x$ is contained in $\Ome^c$. Since $\Refl_{w,b}(\Ome^+_{w,b})$ is the closure of an open set, $x$ is in its interior or in its boundary. We only consider the latter case. We can choose a point $x'$ from the $\ep_1$-neighborhood of $x$ such that $x' \in (\Refl_{w,b}(\Ome^+_{w,b}))^\circ$. There exists an $\ep_2 >0$ such that the $\ep_2$-neighborhood of $x'$ is contained in the interior of $\Refl_{w,b}(\Ome^+_{w,b}) \cap B_{\ep _1}(x)$. Hence the $\ep_2$-neighborhood of $x'$ is contained in the intersection $\Refl_{w,b}(\Ome^+_{w,b}) \cap \Ome^c$, that is, $x'$ is an interior point of the intersection.

Next, we complete the proof. Let $x$ be an interior point of $\Refl_{w,b}(\Ome^+_{w,b}) \cap \Ome^c$, and $\ep$ be a positive constant such that $B_\ep (x) \subset \Refl_{w,b}(\Ome^+_{w,b}) \cap \Ome^c$. Let $\xi = \Refl_{w,b}^{-1} (x)$, then the $\ep$-neighborhood of $\xi$ is contained in $\Ome^+_{w,b}$. The continuity of the map
\[
S^{m-1} \ni u \mapsto \Refl_{u,b} (\xi ) = \xi +2 (b-\xi \cdot u )u \in \R^m
\]
implies the existence of a positive constant $\de$ such that, for any $u \in B_\de (w) \cap S^{m-1}$, the ball $B_{\ep /2} (\xi )$ is contained in $\Ome_{u,b}^+$, and we have
\[
\Refl_{u,b}\( B_{\ep /2}(\xi ) \) \subset B_\ep (x) \subset \Ome^c ,
\]
which completes the proof.
\end{proof}

For a direction $v \in S^{m-1}$, we denote the orthogonal complement vector space by $v^\perp$, that is, we let
\begin{equation}
v^\perp = \left\{ \left. z \in \R^m \rvert z \cdot v =0 \right\} .
\end{equation} 
We understand that $\Ome$ is convex in a direction $v$ if the intersection $\Ome \cap (\Span \langle v \rangle + z)$ is connected for any point $z$ in $v^\perp$.

\begin{prop}\label{uf_symm}
Let $\Ome$ be a body in $\R^m$.
\begin{enumerate}[$(1)$]
\item If there exist $p$ $(1 \leq p \leq m)$ independent directions $v_1,\ldots ,v_p \in S^{m-1}$ such that $\Ome$ is symmetric in the hyperplanes $v_1^\perp,\ldots ,v_p^\perp$ and convex in the directions $v_1 ,\ldots ,v_p$, then we have
\[
Uf(\Ome )\subset \bigcap_{j=1}^{p} v_j^\perp .
\]
\item If the dimension of the minimal unfolded region of $\Ome$ is $p$ $(0 \leq p \leq m-1)$, then there exists a direction $w \in S^{m-1}$ orthogonal to $Uf(\Ome )$ such that $\Ome$ is symmetric in a hyperplane parallel to $w^\perp$ and convex in $w$.
\end{enumerate}
\end{prop}

\begin{proof}
(1) We remark that $l( v_j ) =l( -v_j )$ for any $1 \leq j \leq p$. Let us show $l( v_j )=0$ for any $1 \leq j \leq p$, which implies 
\[
Uf(\Ome ) 
\subset \bigcap_{j=1}^{p} \( \left\{ z \in \R^m \lvert z \cdot v_j \leq l\( v_j \) \right\} \right. \cap \left\{ z \in \R^m \lvert z \cdot \( -v_j \) \leq l\( -v_j \) \right\} \right. \)
= \bigcap_{j=1}^{p} v_j^\perp .
\] 

Suppose that we can choose a number $j$ with $l( v_j ) >0$. There exists a height $b$ ($0< b < l( v_j )$) such that $\Refl_{v_j ,b} ( \Ome_{v_j,b}^+) \cap \Ome^c \neq \emptyset$. We can choose a point $x \in \Ome_{v_j,b}^+$ such that $x' =\Refl_{v_j ,b} (x) \in \Ome^c$. 

On the other hand, from the symmetry of $\Ome$ in the hyperplane $v_j^\perp$, we have $x''=\Refl_{v_j ,0} (x) \in \Ome$. By the convexity of $\Ome$ for $v_j$, we have
\[
\emptyset \neq \overline{xx'} \cap \Ome^c \subset \overline{xx''} \cap \Ome^c \subset \Ome \cap \Ome^c =\emptyset ,
\] 
which is a contradiction.

(2) Since $Uf(\Ome )$ is compact and convex, we may assume that $Uf(\Ome )$ is contained in the $p$-dimensional vector space $\R^p \times \{ 0 \}^{m-p} \subset \R^p \times \R^{m-p} = \R^m$. By a translation of $\R^p \times \{ 0 \}^{m-p}$, we also may assume that the centroid of $Uf(\Ome )$, denoted by $G_\Ome$, coincides with the origin. 

We first show that the minimum value of $l$ is zero. Suppose that $l(v)$ is positive for any $v \in S^{m-1}$. By the lower semi-continuity of $l$, we have
\[
\rho = \inf_{v \in S^{m-1}} l(v)  = \min_{v \in S^{m-1}} l(v)>0 .
\] 
Then the $m$-dimensional ball $B_{\rho}(0)$ is contained in $Uf(\Ome )$, which is a contradiction. Hence there exists a direction $w \in S^{m-1}$ such that $l(w) =0$. 

In order to show the symmetry of $\Ome$ with respect to the hyperplane orthogonal to $w$, we show that $\Ome$ is the union of $\Ome^+_{w,0}$ and $\Refl_{w,0} (\Ome^+_{w,0} )$. Suppose that the set minus $\Ome \sm ( \Ome^+_{w,0} \cup \Refl_{w,0} ( \Ome^+_{w,0} ) )$ is not empty. Since $\Ome$ is a body, $\Ome \sm ( \Ome^+_{w,0} \cup \Refl_{w,0} ( \Ome^+_{w,0} ) )$ has an interior point. By the reflection argument, we have
\[
0> \int_{\Ome \sm \( \Ome^+_{w,0} \cup \Refl_{w,0} \( \Ome^+_{w,0} \) \)} y \cdot w dy 
=\int_\Ome y\cdot w dy
=\Vol (\Ome ) G_\Ome \cdot w
=0,
\] 
which is a contradiction. Hence $\Ome$ is symmetric in the hyperplane $w^\perp$.

Finally, we show the convexity of $\Ome$ in the direction $w$. We assume the existence of two points $x$ and $x'$ in $\Ome$ such that the line segment $\overline{xx'}$ is parallel to the vector space $\Span \langle w \rangle$ and contains a point $\xi$ in $\Ome^c$. We may assume $(x + \xi )\cdot w >0$. Let $b = (( x+\xi ) \cdot w)/2$. Then, we have $\xi = \Refl_{w,b} (x)$, which contradicts to $l(w)=0 <b$.

Furthermore, the first assertion implies that the direction $w$ is orthogonal to the minimal unfolded region of $\Ome$.
\end{proof}
\section{Examples of the graphs}
Let $\Ome$ be a body in $\R^m$ ($m \geq 2$) with a piecewise $C^1$ boundary. In this section, in order to investigate the number of $k$-centers of $\Ome$, using the calculator Maple, we produce some examples of the graphs of the $r^{\al -m}$-potentials
\begin{equation}
V_\Ome^{(\al )}(x)=
\begin{cases}
\ds \sign (m-\al ) \int_\Ome r^{\al -m} dy &(0 < \al \neq m),\\
\ds -\int_\Ome \log r dy &(\al =m ),
\end{cases}
\end{equation}
and its second derivatives. When we use Maple to draw the graph of the $r^{\al -m}$-potential, it is useful to use the boundary integral expression,
\begin{equation}
V_\Ome^{(\al )} (x) =
\begin{cases}
\ds - \frac{\sign (m-\al )}{\al} \int_{\pd \Ome} r^{\al -m} (x-y) \cdot n(y) d\sigma (y) &(0< \al \neq m),\\
\ds \frac{1}{m} \int_{\pd \Ome} \( \log r -\frac{1}{m} \) (x-y) \cdot n(y) d\sigma (y) & (\al =m) ,
\end{cases}
\end{equation}
for any $x \in \R^m \sm \pd \Ome$ (\cite[Theorem 2.8]{O1}).

\begin{ex}\label{ex_discs}
{\rm Let $m=2$ and $\Ome =\{ ( y_1 ,y_2) \vert (y_1 +1)^2 + y_2^2 \leq 1\} \cup \{ (y_1 ,y_2) \vert (y_1 -1)^2 + y_2^2 \leq 1\}$. Then we have
\begin{align*}
\pd \Ome &= \left\{ (\cos \theta -1 ,\sin \theta ) \lvert 0\leq \theta \leq 2\pi \right\} \right. \cup \left\{ (\cos \theta +1 ,\sin \theta ) \lvert 0\leq \theta \leq 2\pi \right\} \right. ,\\ 
Uf(\Ome )&=\left\{ \( y_1 ,0 \) \lvert -1 \leq y_1 \leq 1 \right\} \right. , \\
V_\Ome^{(\al )} (\la , 0) &= -\frac{1}{\al} \int_0^{2\pi} \( \( \la -\cos \theta +1 \)^2 +\sin^2 \theta \)^{(\al -2)/2} \( \( \la +1 \) \cos \theta -1 \) d\theta  \\
&\quad -\frac{1}{\al} \int_0^{2\pi} \( \( \la -\cos \theta -1 \)^2 +\sin^2 \theta \)^{(\al -2)/2} \( \( \la -1 \) \cos \theta -1 \) d\theta , 
\end{align*}
and the graph of the potential $V_\Ome^{(3/2)} (\la ,0)$ for $-1 \leq \la \leq 1$ is Figure \ref{discs}. Hence, in this case, $\Ome$ has two $r^{-1/2}$-centers.}
\end{ex}

\begin{ex}\label{ex_annulus}
{\rm Let $m=2$ and $\Ome =\{ (y_1 ,y_2) \vert 1 \leq y_1^2 + y_2^2 \leq 4 \}$. Then we have
\begin{align*}
\pd \Ome &= \left\{ ( 2 \cos \theta ,2 \sin \theta ) \lvert 0 \leq \theta \leq 2\pi \right\} \right. \cup \( -\left\{ ( \cos \theta ,\sin \theta ) \lvert 0 \leq \theta \leq 2\pi \right\} \right. \) ,\\
Uf(\Ome )&=\left\{ \( y_1 ,y_2 \) \lvert y_1^2 +y_2^2 \leq \frac{9}{4} \right\} \right. ,\\
V_\Ome^{(\al )} (\la ,0) &= \frac{1}{\al} \int_0^{2\pi} \( \la^2 -2\la \cos \theta + 1 \)^{(\al -2)/2} \( \la \cos \theta -1 \) d\theta \\
&\quad -\frac{2}{\al} \int_0^{2\pi} \( \la^2 -4 \la \cos \theta +4 \)^{(\al -2)/2} \( \la \cos \theta -2 \) d\theta ,
\end{align*}
and the graph of the potential $V_\Ome^{(3/2)} (\la ,0)$ for $-3/2 \leq \la \leq 3/2$ is Figure \ref{annulus}. Hence, in this case, the set of $r^{-1/2}$-centers of $\Ome$ is a circle.}
\end{ex}

\begin{ex}[{\cite[Remark 3.13]{O1}}]\label{ex_triangle}
{\rm Let $m=2$ and $\Ome = \{ (y_1 ,y_2) \vert 0 \leq y_1 \leq 1,\ 0\leq \vert y_2 \vert \leq (\tan (\pi/10)) y_1 \}$. Then we have
\begin{align*}
\pd \Ome &= \left\{ \( y_1 ,y_2 \) \lvert 0\leq y_1 \leq 1, \ y_2 = - \( \tan \frac{\pi}{10} \) y_1 \right\} \right. \cup \left\{ \( 1, y_2\) \lvert -\tan \frac{\pi}{10} \leq y_2 \leq \tan \frac{\pi}{10} \right\} \right. ,\\
&\quad \cup \(- \left\{ \( y_1 ,y_2 \) \lvert 0\leq y_1 \leq 1, \ y_2 = \( \tan \frac{\pi}{10} \) y_1 \right\} \right. \) ,\\
\( \frac{1}{2},0 \) &\in Uf(\Ome ) \subset \left\{ \( y_1 ,0 \) \lvert \frac{1}{2} \leq y_1 \leq 1 \right\} \right. ,\\
\frac{\pd^2 V_\Ome^{(\al )}}{\pd x_1^2} (\la ,0) &= -2(2-\al )\tan \frac{\pi}{10} \int_0^1 \( \( \la -t \)^2 + \( \( \tan \frac{\pi}{10} \) t\)^2 \)^{(\al -4)/2} \( \la -t \) dt\\
&\quad +2(2-\al ) (\la -1) \int_0^{\tan \frac{\pi}{10}} \( (\la -1)^2 +t^2 \)^{(\al -4)/2} dt ,
\end{align*}
and the graph of the second derivative of the potential $V_\Ome^{(3/2)} (\la ,0)$ for $0\leq \la \leq 1$ is Figure \ref{isosceles1}. Moreover, the contribution of the slopes to the boundary integral (the first integral) is Figure \ref{isosceles2}. Hence, in this case, $\Ome$ has a unique $r^{-1/2}$-center.}
\end{ex}

\begin{ex}\label{ex_cone}
{\rm Let $m=3$ and $\Ome = \{ (y_1 ,y_2,y_3) \vert 0 \leq y_1 \leq 1,\ y_2^2 + y_3^2 \leq (\tan^2 (\pi /10)) y_1^2 \}$. Then we have
\begin{align*}
\pd \Ome &= \left. \left\{ \( t, \( \tan \frac{\pi}{10} \) t \cos \theta , \( \tan \frac{\pi}{10} \) t \sin \theta \) \rvert 0\leq t \leq 1,\ 0\leq \theta \leq 2\pi \right\}  \\
&\quad \cup \left\{ (1, r \cos \theta , r\sin \theta ) \lvert 0 \leq r \leq \tan \frac{\pi}{10}, \ 0 \leq \theta \leq 2 \pi \right\} \right. ,\\
\( \frac{1}{2},0,0\) &\in Uf(\Ome ) \subset \left\{ \( y_1 ,0,0 \) \lvert \frac{1}{2} \leq y_1 \leq 1 \right\} \right. ,\\
\frac{\pd^2 V_\Ome^{(\al )}}{\pd x_1^2}(\la ,0,0) &= -2\pi (3-\al ) \tan^2 \frac{\pi}{10} \int_0^1 \( \( \la -t \)^2 +\( \( \tan \frac{\pi}{10} \) t \)^2 \)^{(\al -5)/2} (\la -t )tdt \\
&\quad + 2\pi (3-\al ) (\la -1) \int_0^{\tan \frac{\pi}{10}} \( \( \la -1 \)^2 +r^2 \)^{(\al -5)/2} rdr ,
\end{align*}
and the graph of the second derivative of the potential $V_\Ome^{(5/2)} (\la ,0,0)$ for $0\leq \la \leq 1$ is Figure \ref{cone1}. Moreover, the contribution of the side to the boundary integral (the first integral) is Figure \ref{cone2}. Hence, in this case, $\Ome$ has a unique $r^{-1/2}$-center.}
\end{ex} 
\begin{figure}[htbp]
\begin{center}
\begin{minipage}[h]{0.45\linewidth}
\scalebox{0.4}{\includegraphics[clip]{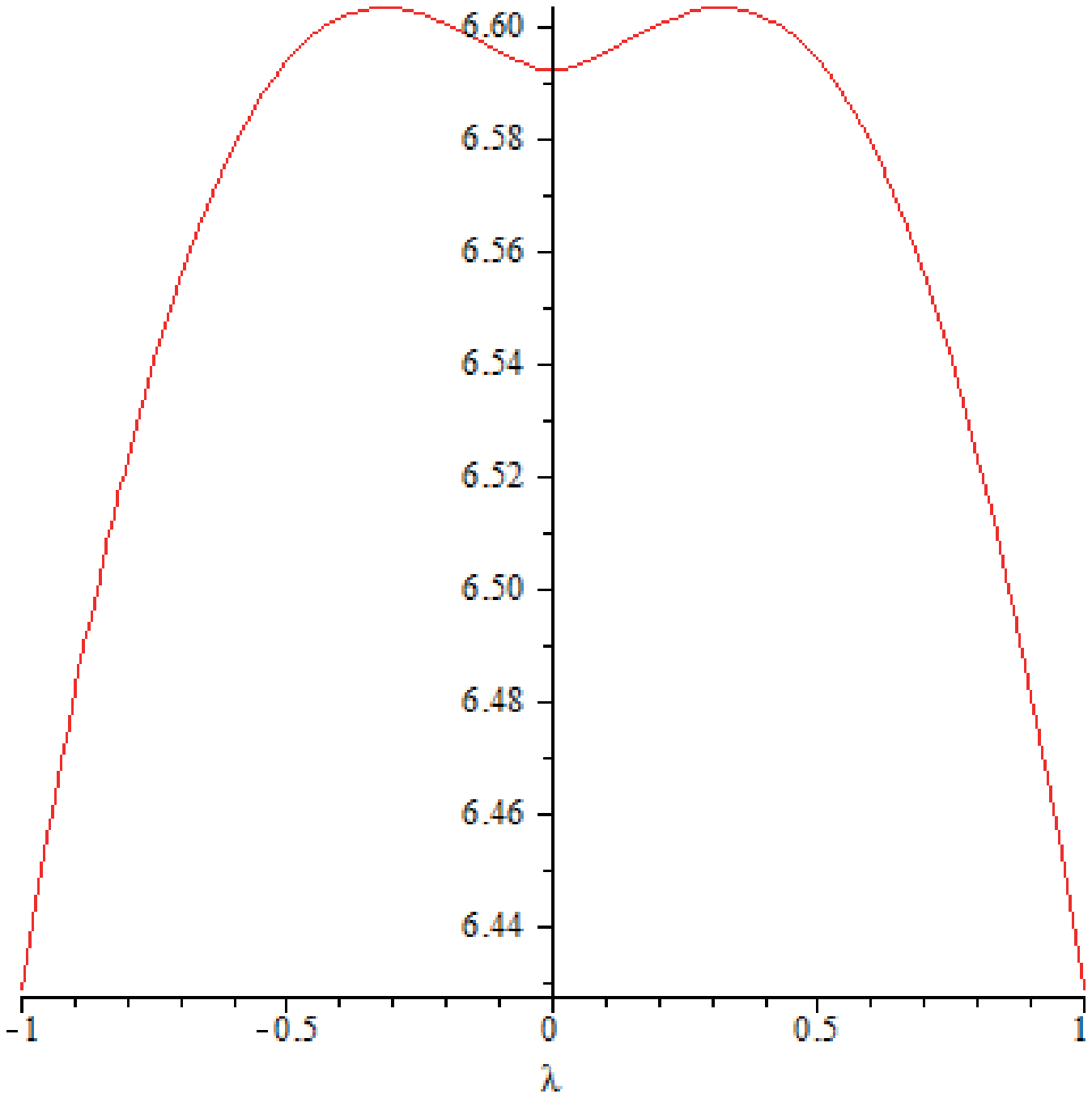}}
\caption{The graph of $V_\Ome^{(3/2)} (\la ,0)$ when $\Ome$ is the union of two discs}
\label{discs}
\end{minipage}
\hspace{0.01\linewidth}
\begin{minipage}[h]{0.45\linewidth}
\scalebox{0.4}{\includegraphics[clip]{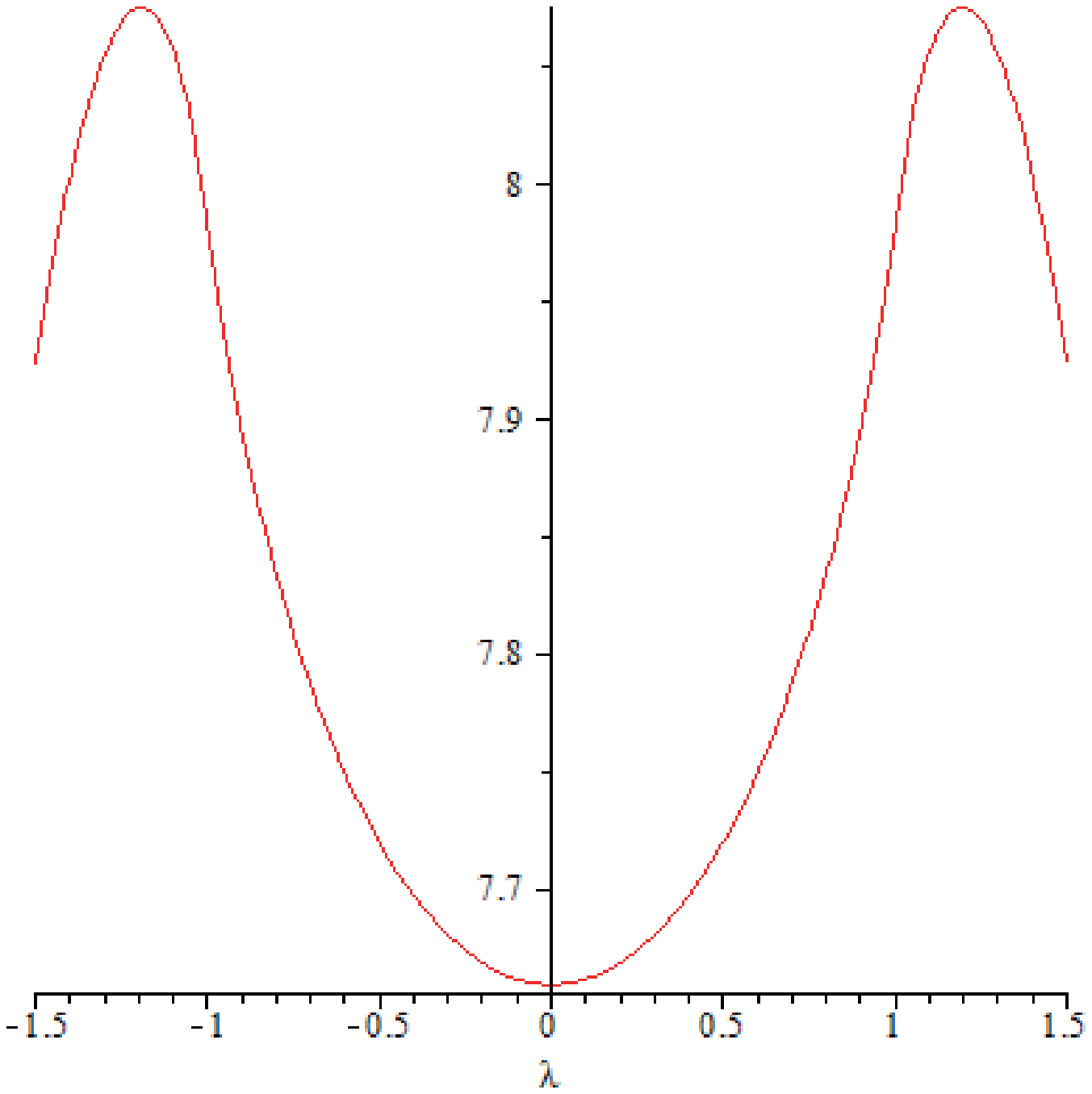}}
\caption{The graph of $V_\Ome^{(3/2)} ( \la ,0)$ when $\Ome$ is an annulus}
\label{annulus}
\end{minipage}
\begin{minipage}[h]{0.45\linewidth}
\scalebox{0.4}{\includegraphics[clip]{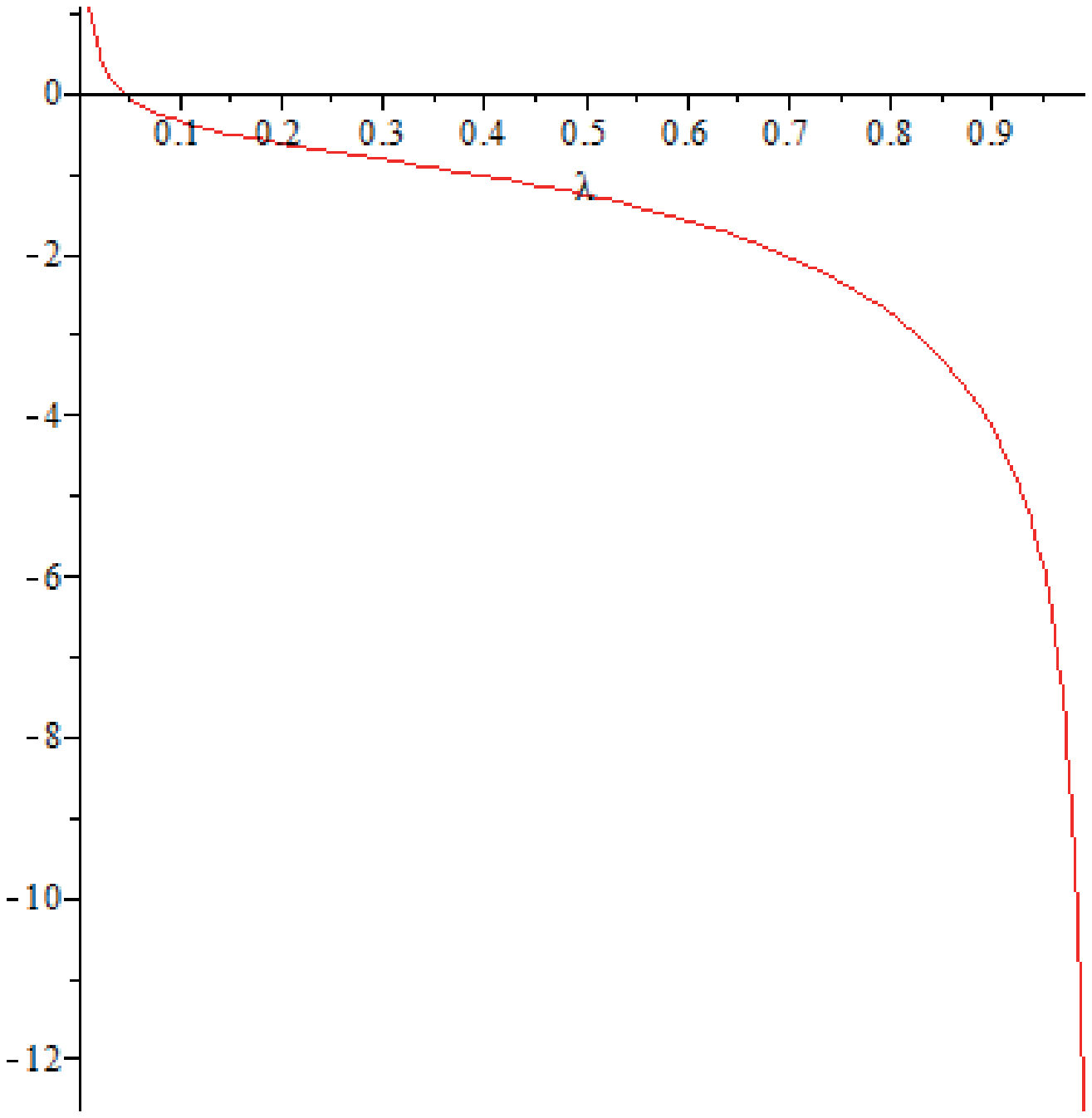}}
\caption{The graph of $(\pd^2 V_\Ome^{(3/2)} / \pd x_1^2) (\la ,0)$ when $\Ome$ is an isosceles triangle}
\label{isosceles1}
\end{minipage}
\hspace{0.01\linewidth}
\begin{minipage}[h]{0.45\linewidth}
\scalebox{0.4}{\includegraphics[clip]{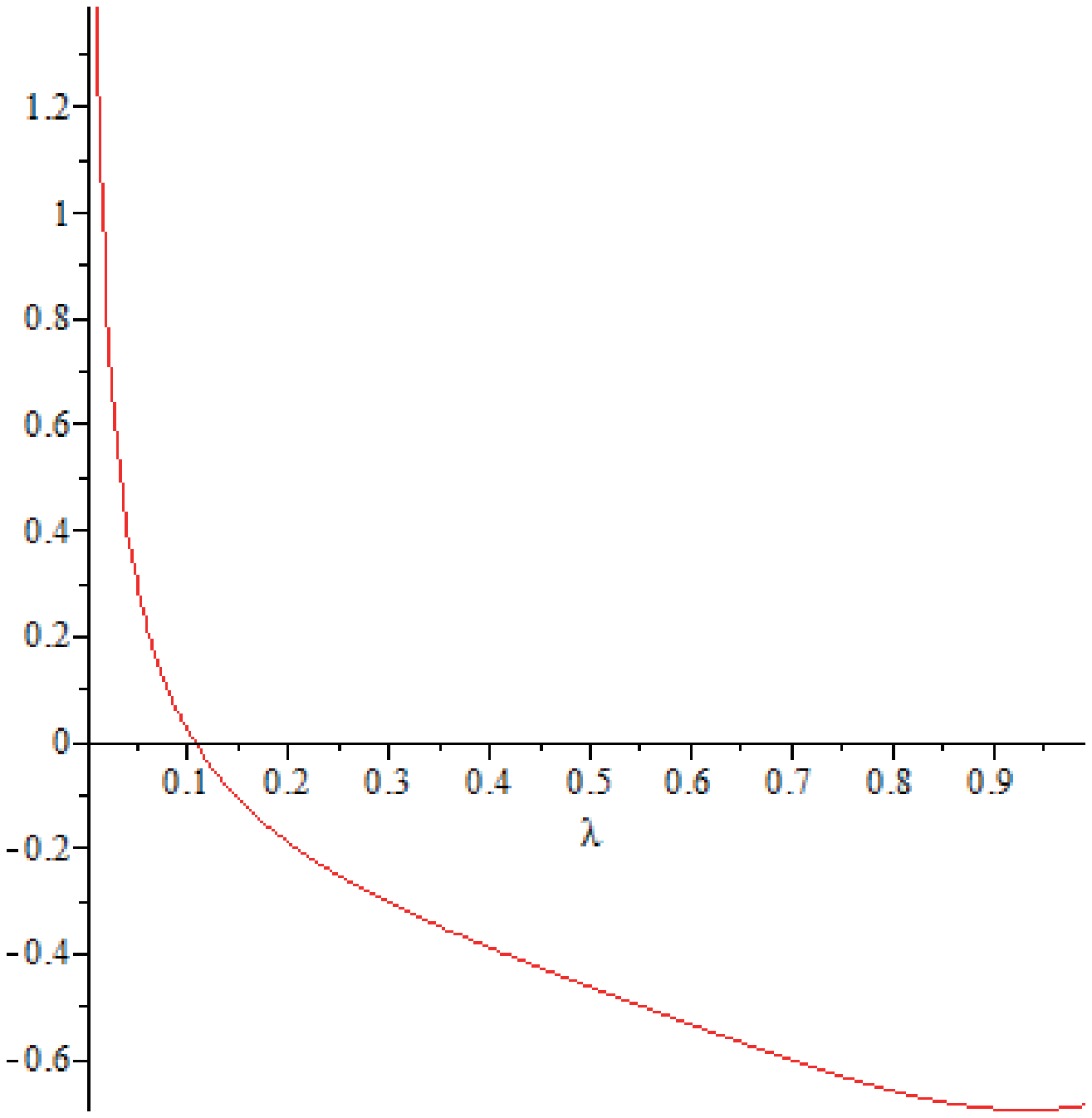}}
\caption{The contribution of the slopes to the boundary integral}
\label{isosceles2}
\end{minipage}
\begin{minipage}[h]{0.45\linewidth}
\scalebox{0.4}{\includegraphics[clip]{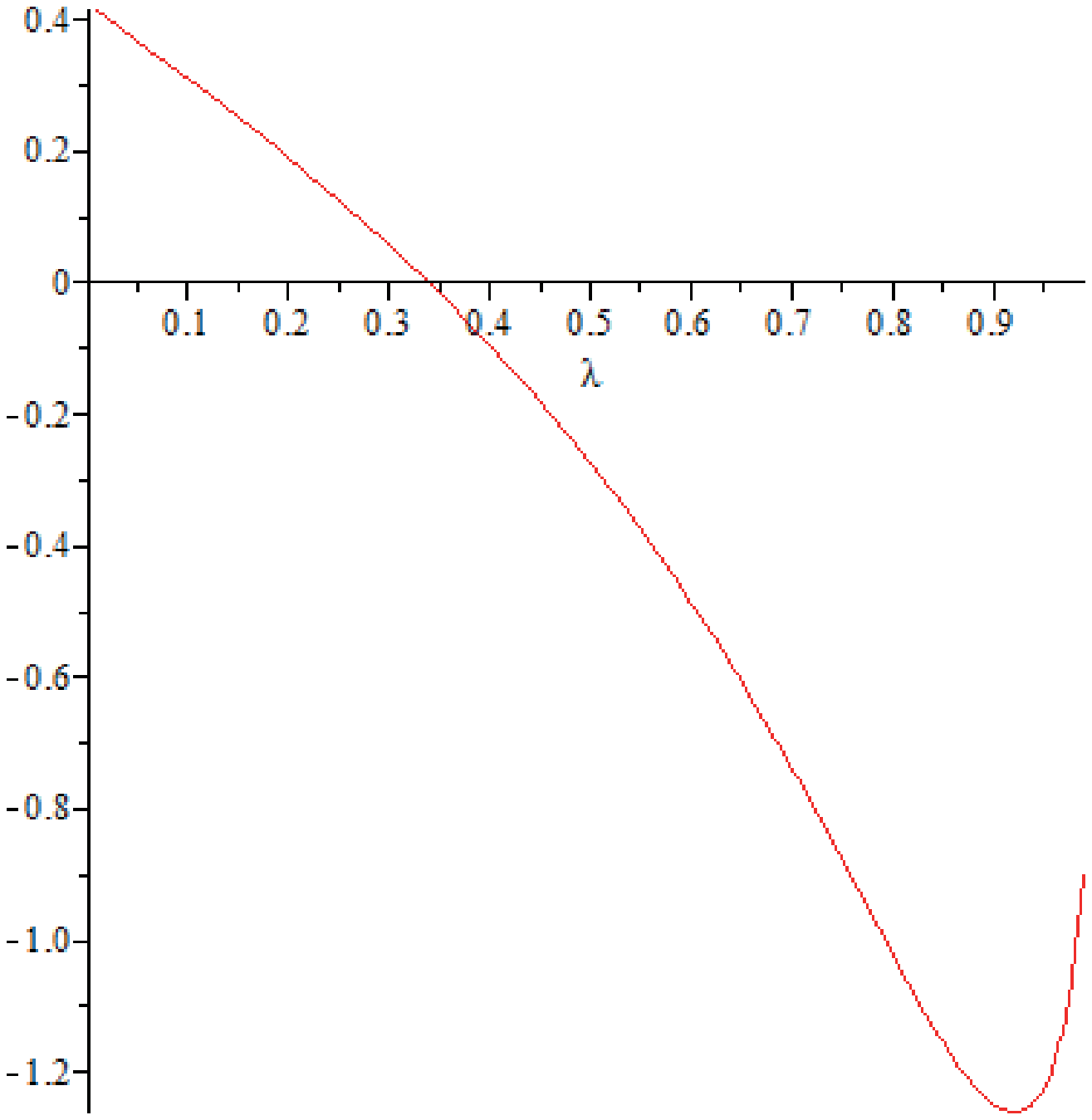}}
\caption{The graph of $(\pd^2 V_\Ome^{(5/2)} / \pd x_1^2) (\la ,0)$ when $\Ome$ is a cone}
\label{cone1}
\end{minipage}
\hspace{0.01\linewidth}
\begin{minipage}[h]{0.45\linewidth}
\scalebox{0.4}{\includegraphics[clip]{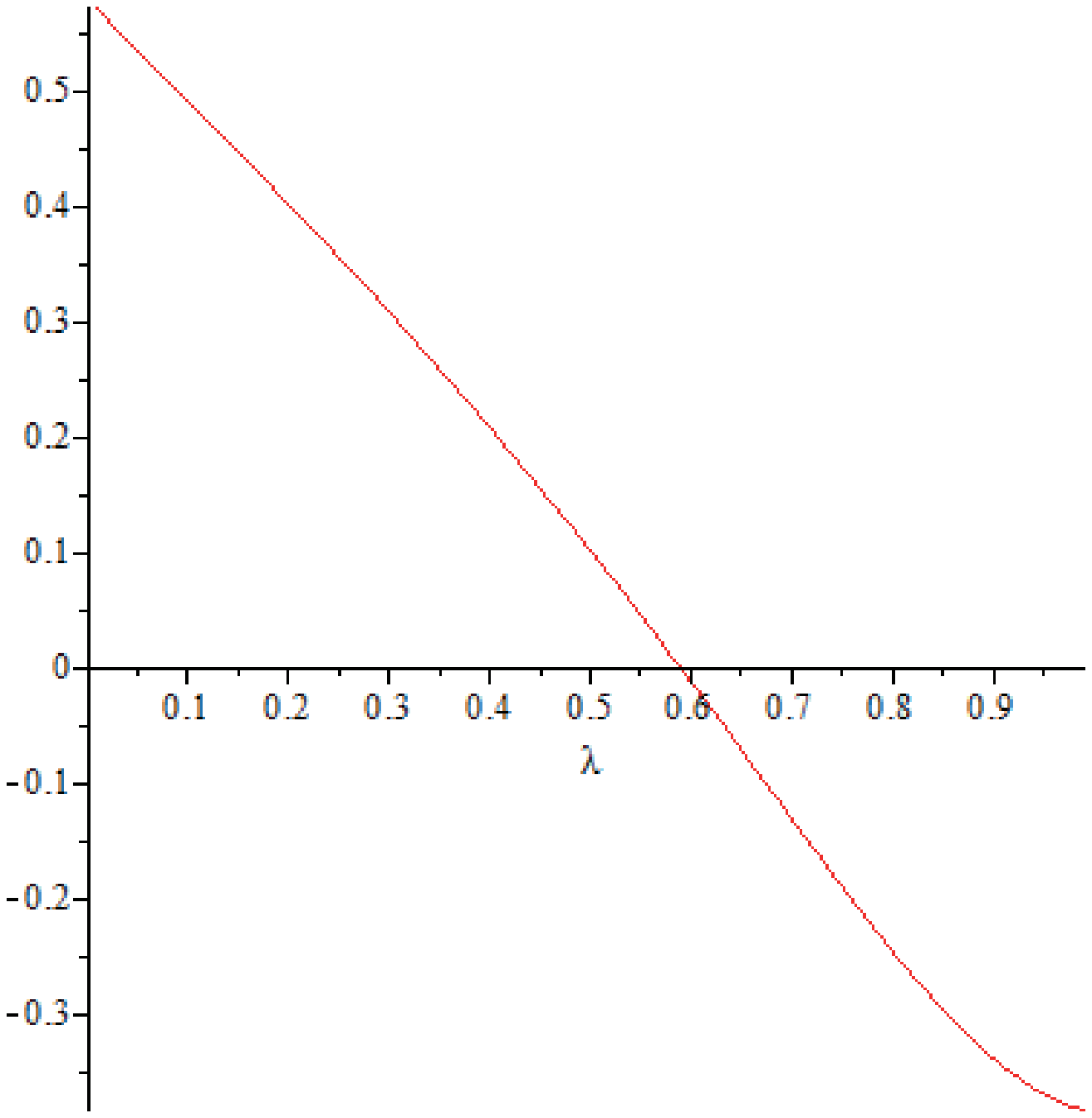}}
\caption{The contribution of the side to the boundary integral}
\label{cone2}
\end{minipage}
\end{center}
\end{figure}

\begin{ex}\label{ex_arcon}
{\rm Let $m=3$ and $\Ome = \{( y_1 ,y_2,y_3) \vert 0 \leq y_1 \leq 1,\ y_2^2 + y_3^2 \leq (\tan^2 (\pi /10)) y_1\}$. Then we have
\begin{align*}
\pd \Ome &= \left. \left\{ \( t, \( \tan \frac{\pi}{10} \) \sqrt{t} \cos \theta , \( \tan \frac{\pi}{10} \) \sqrt{t} \sin \theta \) \rvert 0\leq t \leq 1,\ 0\leq \theta \leq 2\pi \right\}  \\
&\quad \cup \left\{ (1, r \cos \theta , r\sin \theta ) \lvert 0 \leq r \leq \tan \frac{\pi}{10}, \ 0 \leq \theta \leq 2 \pi \right\} \right. ,\\
\( \frac{1}{2},0,0\) &\in Uf(\Ome ) \subset \left\{ \( y_1 ,0,0 \) \lvert \frac{1}{2} \leq y_1 \leq 1 \right\} \right. ,\\
\frac{\pd^2 V_\Ome^{(\al )}}{\pd x_1^2}(\la ,0,0) &= -\pi (3-\al ) \tan^2 \frac{\pi}{10} \int_0^1 \( \( \la -t \)^2 + \( \tan^2 \frac{\pi}{10} \) t \)^{(\al -5)/2} (\la -t )dt \\
&\quad + 2\pi (3-\al ) (\la -1) \int_0^{\tan \frac{\pi}{10}} \( \( \la -1 \)^2 +r^2 \)^{(\al -5)/2} rdr ,
\end{align*}
and the graph of the second derivative of the potential $V_\Ome^{(5/2)} (\la ,0,0)$ for $0\leq \la \leq 1$ is Figure \ref{arcon1}. Moreover, the contribution of the side to the boundary integral (the first integral) is Figure \ref{arcon2}. Hence, in this case, $\Ome$ has a unique $r^{-1/2}$-center.}
\end{ex}
\begin{figure}[htbp]
\begin{center}
\begin{minipage}[h]{0.45\linewidth}
\scalebox{0.4}{\includegraphics[clip]{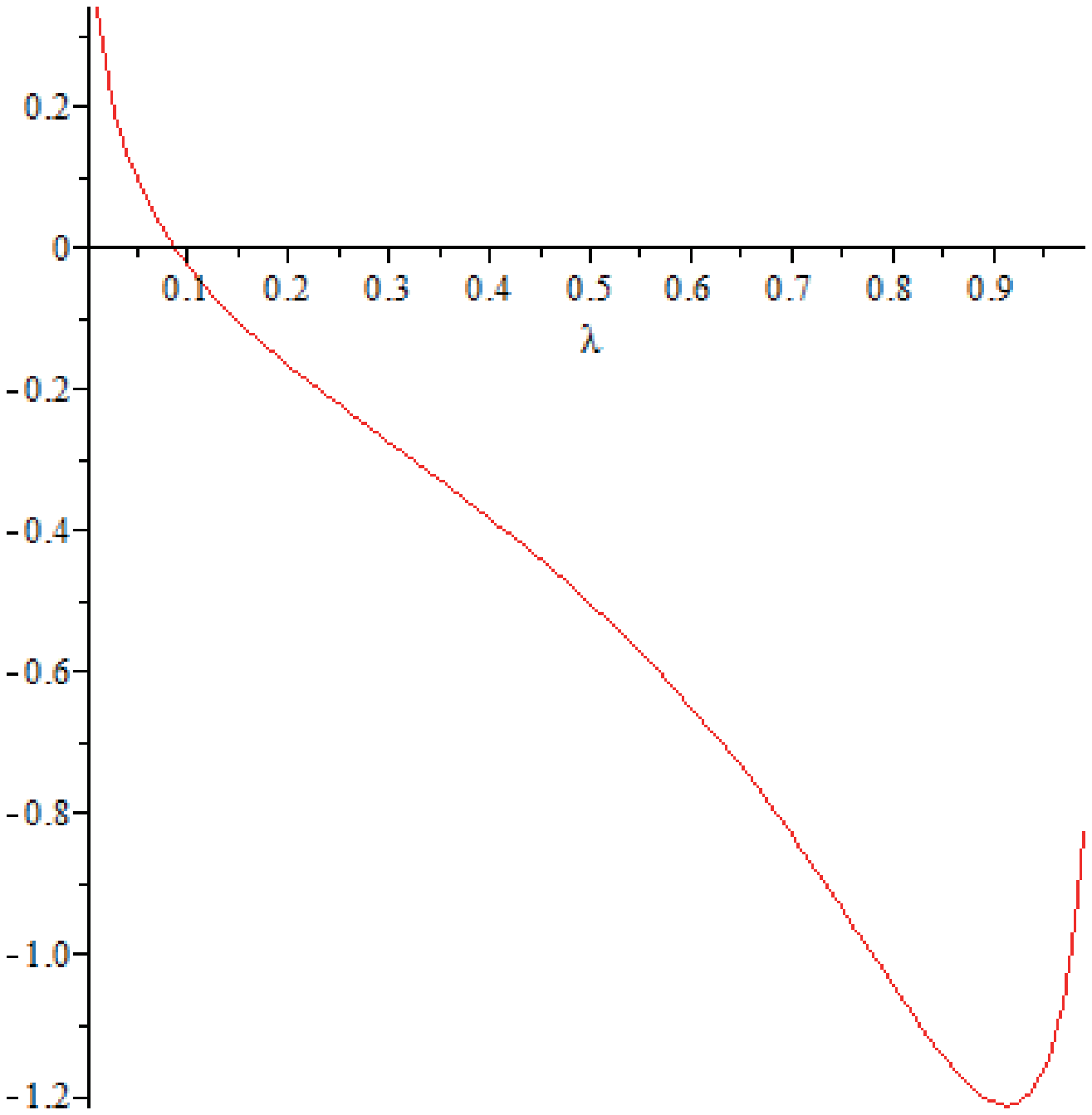}}
\caption{The graph of $(\pd^2 V_\Ome^{(5/2)} / \pd x_1^2) (\la ,0)$ when $\Ome$ is a solid of revolution of a parabola}
\label{arcon1}
\end{minipage}
\hspace{0.01\linewidth}
\begin{minipage}[h]{0.45\linewidth}
\scalebox{0.4}{\includegraphics[clip]{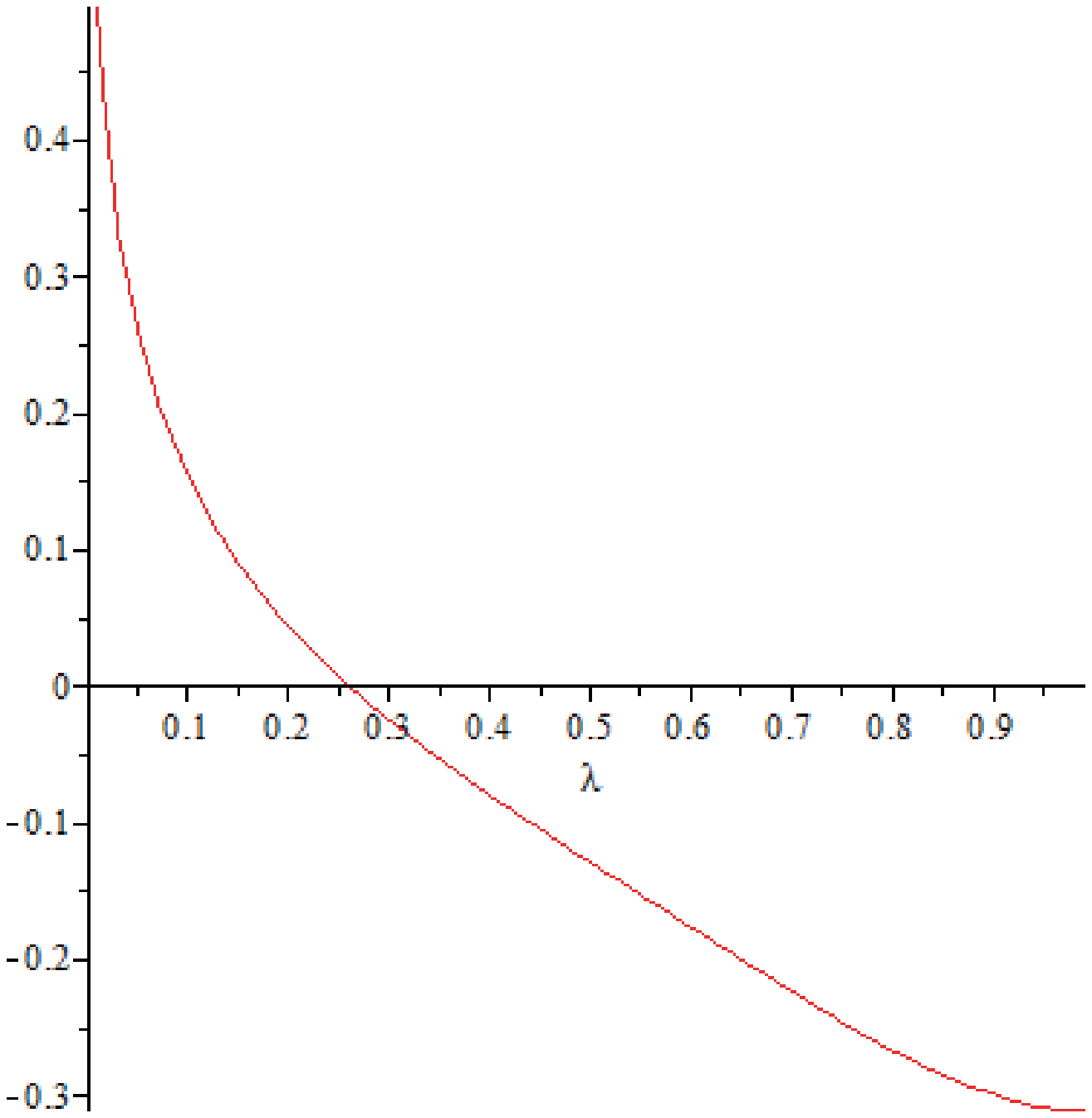}}
\caption{The contribution of the side to the boundary integral}
\label{arcon2}
\end{minipage}
\end{center}
\end{figure}

\section{Uniqueness of a {\boldmath $k$}-center}
Let $\Ome$ be a body in $\R^m$. In this section, we investigate the uniqueness of a $k$-center of $\Ome$. Put 
\begin{equation}
d(\Ome ) = \min \left\{ \vert z - w \vert \lvert z \in Uf(\Ome ), \ w \in \pd \Ome \right. \right\}, \ D(\Ome ) = \max \left\{ \vert z - w \vert \lvert z \in Uf(\Ome ), \ w \in \pd \Ome \right. \right\}.
\end{equation} 
\subsection{Uniqueness of a center of a suitable axially symmetric body}

\begin{thm}\label{concavity_revolution} 
Let $\ome :[0,1] \to [0,+\infty )$ be a piecewise $C^1$ function such that the function $\ome^{m-1} : t \mapsto \ome(t)^{m-1}$ is concave. Let 
\[
\Ome = \left\{ y= \( y_1 ,\bar{y} \) \in \R \times \R^{m-1} \lvert 0\leq y_1 \leq 1,\ \lvert \bar{y} \rvert \leq \ome \( y_1 \) \right\} \right. .
\]
Suppose that the kernel $k$ is strictly decreasing and satisfies the condition $(C^1_\al )$ for some $\al >1$. If $k'(r)/r$ is increasing on the interval $( d( \Ome ) ,D( \Ome ) )$, then the potential $K_\Ome$ is strictly concave on the minimal unfolded region.
\end{thm}

\begin{proof}
Put
\[
a=\min \left\{ t \in [0,1] \lvert \ome (t)= \max_{0 \leq \tau \leq 1}\ome (\tau ) \right. \right\}, \ 
b=\max \left\{ t \in [0,1] \lvert \ome (t)= \max_{0 \leq \tau \leq 1}\ome (\tau ) \right. \right\}.
\]
Proposition \ref{uf_symm} and the concavity of $\ome$ imply that $Uf(\Ome )$ is contained in the line segment 
\[
\left\{ \( y_1,0 \) \in \R \times \R^{m-1} \lvert  \frac{a}{2} \leq y_1 \leq \frac{1+b}{2} \right\} \right. .
\]
Therefore, we show the negativity of $(\pd^2 K_\Ome / \pd x_1^2) (\la ,0 )$ for any $a/2 \leq \la \leq (1+b)/2$.

By Proposition \ref{regularity}, we have
\begin{align*}
\frac{\pd^2 K_\Ome}{\pd x_1^2} (\la ,0) 
&= - \int_{\pd \Ome} \frac{k'\( \sqrt{ \( \la -y_1 \)^2 + \lvert \bar{y} \rvert^2} \)}{\sqrt{ \( \la -y_1 \)^2 + \lvert \bar{y} \rvert^2}} \( \la -y_1 \) e_1 \cdot n(y) d \sigma (y)\\
&= \la \sigma_{m-2} \( S^{m-2} \) \int_0^{\ome (0)} \frac{k'\( \sqrt{ \la^2 +r^2} \)}{\sqrt{\la^2 + r^2}} r^{m-2}dr\\
&\quad + \frac{\sigma_{m-2} \( S^{m-2} \)}{m-1} \int_0^1 \frac{k'\( \sqrt{ \( \la -t \)^2 +\ome (t)^2} \)}{\sqrt{ \( \la -t \)^2 + \ome (t)^2}} \( \la -t \) d \ome (t)^{m-1} \\
&\quad -(\la -1) \sigma_{m-2} \( S^{m-2} \) \int_0^{\ome (1)} \frac{k'\( \sqrt{(\la -1)^2 +r^2} \)}{\sqrt{(\la -1)^2 + r^2}} r^{m-2}dr .
\end{align*}
For any $a/2 \leq \la \leq (1+b)/2$, the first and third terms are obviously negative. Therefore, it is sufficient to show the negativity of the second integral.

We first consider the case of $a/2 \leq \la \leq a$. We decompose the second integral into 
\[
\(\int_0^{2\la -a}+\int_{2\la -a}^{\la} + \int_{\la}^a \) \frac{k'\( \sqrt{ \( \la -t \)^2 +\ome (t)^2} \)}{\sqrt{ \( \la -t \)^2 + \ome (t)^2}} \( \la -t \) d \ome (t)^{m-1}.
\]
For any $0 \leq \de \leq  a-\la$, the concavity of $\ome^{m-1}$ implies $0 \leq ( \ome^{m-1})'(\la + \de ) \leq (\ome^{m-1})'(\la - \de )$, and the increasing behavior of $\ome$ implies $0 \leq \ome (\la - \de ) \leq \ome (\la +\de )$. Hence we obtain
\begin{align*}
&\quad \( \int_{2\la -a}^{\la}+\int_{\la}^a \) \frac{k'\( \sqrt{(\la - t)^2 + \ome (t)^2}\)}{\sqrt{(\la - t)^2 + \ome (t)^2}} (\la -t) d \ome (t)^{m-1} \\
&=\int_0^{a-\la} \( \frac{k'\( \sqrt{(\de^2 +\ome (\la -\de )^2}\)}{\sqrt{\de^2 +\ome (\la -\de )^2}} \( \ome^{m-1} \) '(\la -\de ) - \frac{k'\( \sqrt{\de^2 +\ome ( \la + \de )^2} \)}{\sqrt{\de^2 +\ome ( \la + \de )^2}} \( \ome^{m-1} \) '(\la +\de )  \) \de d\de \\
&\leq 0.
\end{align*}
Furthermore, we can easily get
\[
\int_0^{2\la -a} \frac{k'\( \sqrt{ \( \la -t \)^2 +\ome (t)^2} \)}{\sqrt{ \( \la -t \)^2 + \ome (t)^2}} \( \la -t \)  d\ome (t)^{m-1} <0,
\]
which completes the proof in the case of $a/2 \leq \la \leq a$.

The same argument works for the case of $b \leq \la \leq (1+b)/2$. Furthermore, the negativity of $(\pd^2 K_\Ome / \pd x_1^2) (\la ,0 )$ for $a \leq \la \leq b$ is obvious.
\end{proof}

\begin{cor}\label{uniqueness_revolution} 
Let $\Ome$ and $k$ be as in Theorem \ref{concavity_revolution}. Then $\Ome$ has a unique $k$-center.
\end{cor}

\begin{rem}
{\rm When $\ome (t) =t^p$, the assumption ``$\ome^{m-1}$ is concave'' corresponds to $0 \leq p \leq 1/(m-1)$.}
\end{rem}

\begin{rem}
{\rm In the proof of Theorem \ref{concavity_revolution}, in order to show the negativity of $(\pd^2 K_\Ome / \pd x_1^2) (\la ,0 )$, we decomposed the boundary integral expression of $(\pd^2 K_\Ome / \pd x_1^2) (\la ,0 )$ into the three integrals over the left base, the side and the right base. The integrals over the bases were obviously negative, and we showed the negativity of the integral over the side. 

Unfortunately, this argument does not work for any axially symmetric convex body $\Ome$. When we apply this argument to the cone as in Example \ref{ex_cone}, the boundary integral over the side is not negative on the minimal unfolded region. In other words, in order to show the negativity of $(\pd^2 K_\Ome / \pd x_1^2) (\la ,0 )$ for any axially symmetric convex body $\Ome$, we have to estimate the boundary integral over the bases in more detail. We could not do it and leave the following problem as a conjecture:
\begin{quote}
Does an axially symmetric convex body $\Ome$ have a unique $k$-center? More generally, does a convex body $\Ome$ have a unique $k$-center? We allow to assume some conditions for the kernel $k$ if necessary.
\end{quote}
}
\end{rem}
\subsection{Uniqueness of a center of a non-obtuse triangle}
\begin{thm}\label{concavity_triangle}
Let $\Ome$ be a non-obtuse triangle in $\R^2$. Suppose that the kernel $k$ is strictly decreasing satisfies the condition $(C^1_\al )$ for some $\al >1$. If $k'(r)/r$ is increasing on the interval $( d(\Ome ),D( \Ome ) )$, then the potential $K_\Ome$ is strictly concave on the minimal unfolded region of $\Ome$.
\end{thm}

\begin{proof} 
For an angle $-\pi/2 \leq \theta \leq \pi /2$, let
\[
R_\theta = \(
\begin{array}{cc}
\cos \theta &-\sin \theta \\
\sin \theta  &\cos \theta 
\end{array}
\) .
\]
We show that the second derivative $\pd^2 K_{R_\theta \Ome} / \pd x_1^2$ is negative on the minimal unfolded region of $R_\theta \Ome$ for any $-\pi /2 \leq \theta \leq \pi /2$.

Let $O$ be origin, $P$ the point $(1,0)$, and $Q$ a point $(a,b)$ with the following conditions:
\[
\frac{1}{2} \leq a \leq 1,\ b >0,\ \( a- \frac{1}{2} \)^2 + b^2 \geq \frac{1}{4} .
\]
By an orthogonal action of $\R^2$, we may assume that $\Ome$ is given by $\triangle OPQ$. Let $A$, $B$ and $C$ be the middle points of the line segments $OP$, $PQ$ and $QO$, respectively. We remark that the minimal unfolded region of $\Ome$ is contained in $\triangle ABC$ (see Example \ref{uf_triangle}). 

We identify the notation $z_j$ for the $j$-th coordinate with the function $z_j : \R^2 \ni (z_1,z_2) \mapsto z_j \in \R$. We denote the point $R_\theta P$ by $P_\theta$, for short, and so on. 

We have to consider the following eleven cases about the position of $R_\theta \Ome$ (see Figure \ref{triangleI1} to \ref{triangleII32}):

\begin{description}
\item[Case I] The rotation angle $\theta$ is non-negative.
\begin{description}
\item[Case I.1] $z_1 ( A_\theta ) \leq z_1 ( Q_\theta ) \leq z_1 ( B_\theta )$.
\item[Case I.2] $z_1 ( Q_\theta ) \leq z_1 ( A_\theta ) \leq z_1 ( B_\theta )$.
\item[Case I.3.1] $0 \leq z_1 ( B_\theta ) \leq z_1 ( A_\theta )$ and the slope of the line $P_\theta Q_\theta$ is non-positive.
\item[Case I.3.2] $0 \leq z_1 ( B_\theta ) \leq z_1 ( A_\theta )$ and the slope of the line $P_\theta Q_\theta$ is non-negative.
\item[Case I.4.1] $z_1 ( B_\theta ) \leq 0 \leq z_1 ( A_\theta )$ and the slope of the line $P_\theta Q_\theta$ is non-positive.
\item[Case I.4.2] $z_1 ( B_\theta ) \leq 0 \leq z_1 ( A_\theta )$ and the slope of the line $P_\theta Q_\theta$ is non-negative.
\end{description}
\item[Case II] The rotation angle $\theta$ is non-positive.
\begin{description}
\item[Case II.1] $z_1 ( C_\theta ) \leq z_1 ( A_\theta )$.
\item[Case II.2.1] $z_1 ( A_\theta ) \leq z_1 ( C_\theta ) \leq z_1 ( P_\theta )$ and the slope of the line $OQ_\theta$ is non-negative.
\item[Case II.2.2] $z_1 ( A_\theta ) \leq z_1 ( C_\theta ) \leq z_1 ( P_\theta )$ and the slope of the line $OQ_\theta$ is non-positive.
\item[Case II.3.1] $z_1 ( P_\theta ) \leq z_1 ( C_\theta )$ and the slope of the line $OQ_\theta$ is non-negative.
\item[Case II.3.2] $z_1 ( P_\theta ) \leq z_1 ( C_\theta )$ and the slope of the line $OQ_\theta$ is non-positive.
\end{description}
\end{description}
\begin{figure}[htbp]
\begin{minipage}[h]{0.315\linewidth}
\includegraphics[width=\linewidth]{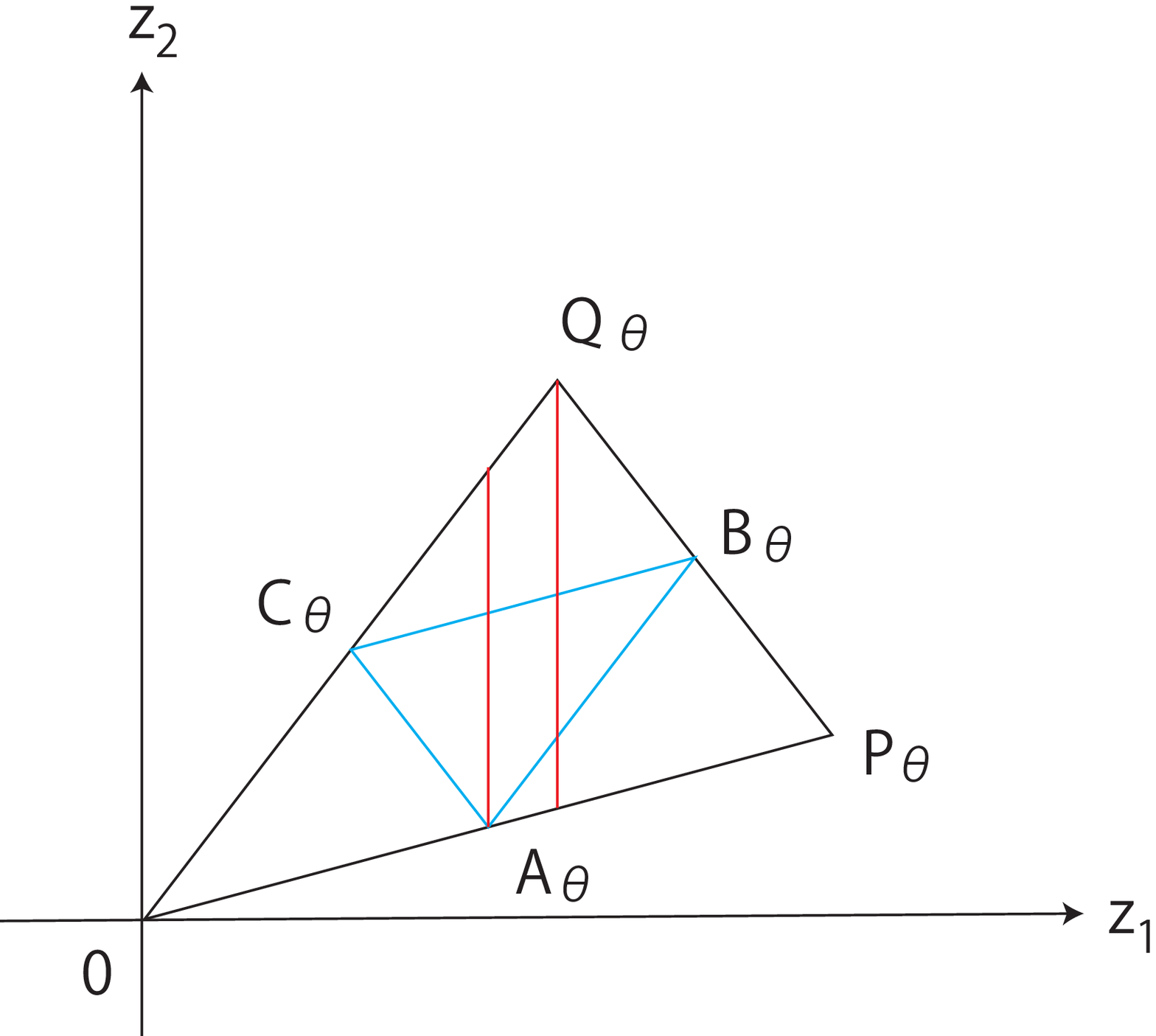}
\caption{Case I.1}
\label{triangleI1}
\end{minipage}
\hspace{0.01\linewidth}
\begin{minipage}[h]{0.315\linewidth}
\includegraphics[width=\linewidth]{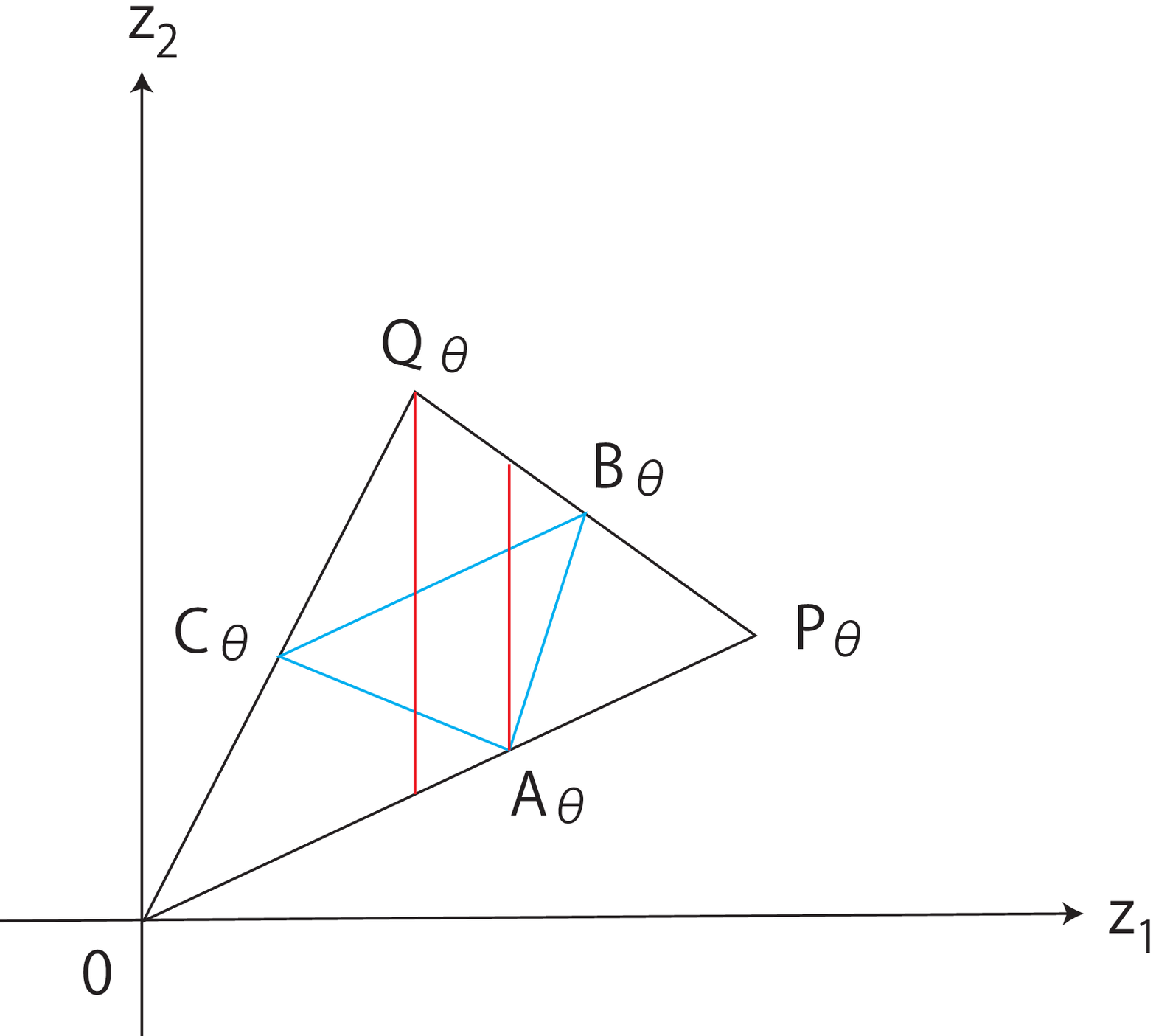}
\caption{Case I.2}
\label{triangleI2}
\end{minipage}
\hspace{0.01\linewidth}
\begin{minipage}[h]{0.315\linewidth}
\includegraphics[width=\linewidth]{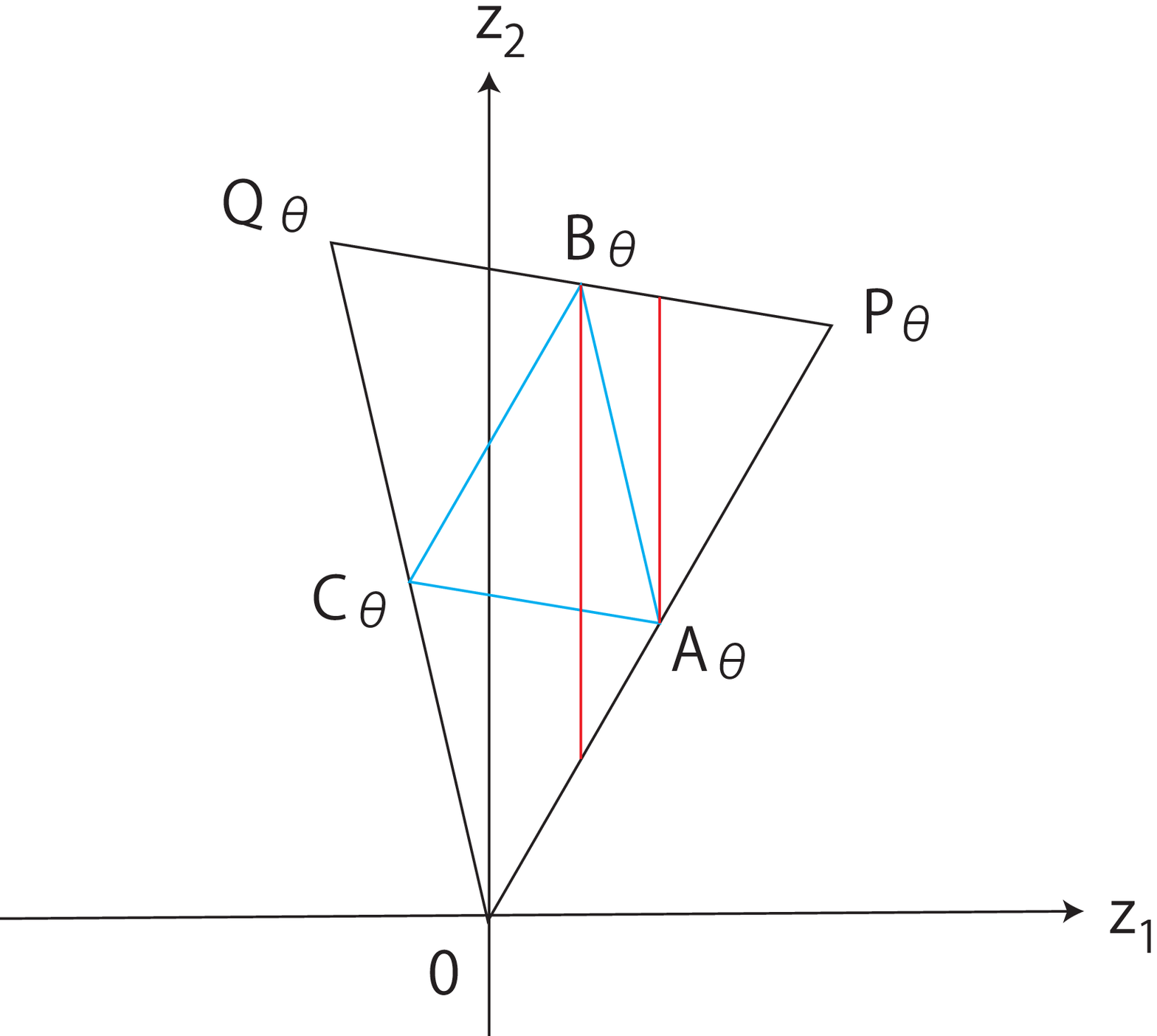}
\caption{Case I.3.1}
\label{triangleI31}
\end{minipage}
\hspace{0.01\linewidth}
\begin{minipage}[h]{0.315\linewidth}
\includegraphics[width=\linewidth]{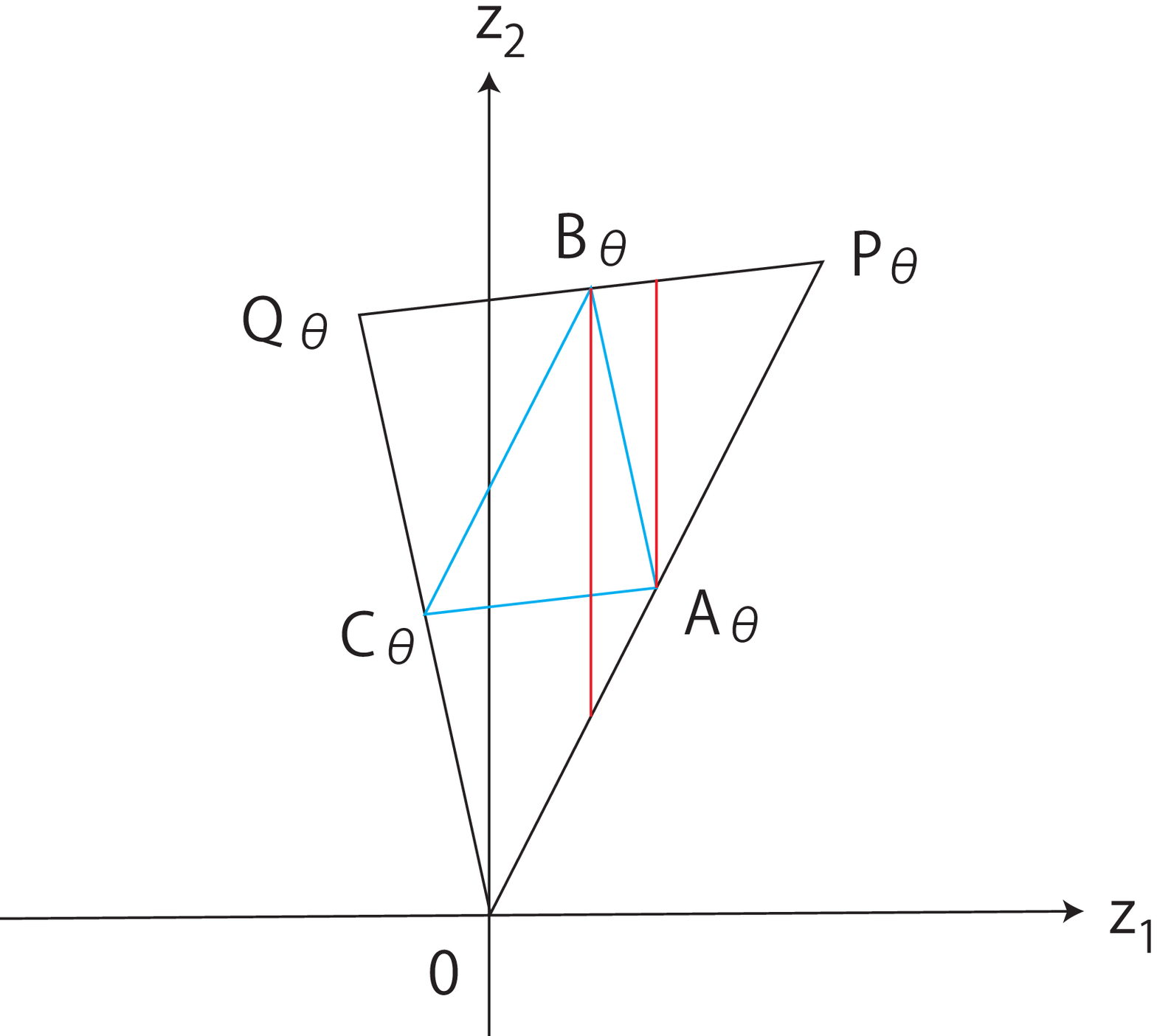}
\caption{Case I.3.2}
\label{triangleI32}
\end{minipage}
\hspace{0.01\linewidth}
\begin{minipage}[h]{0.315\linewidth}
\includegraphics[width=\linewidth]{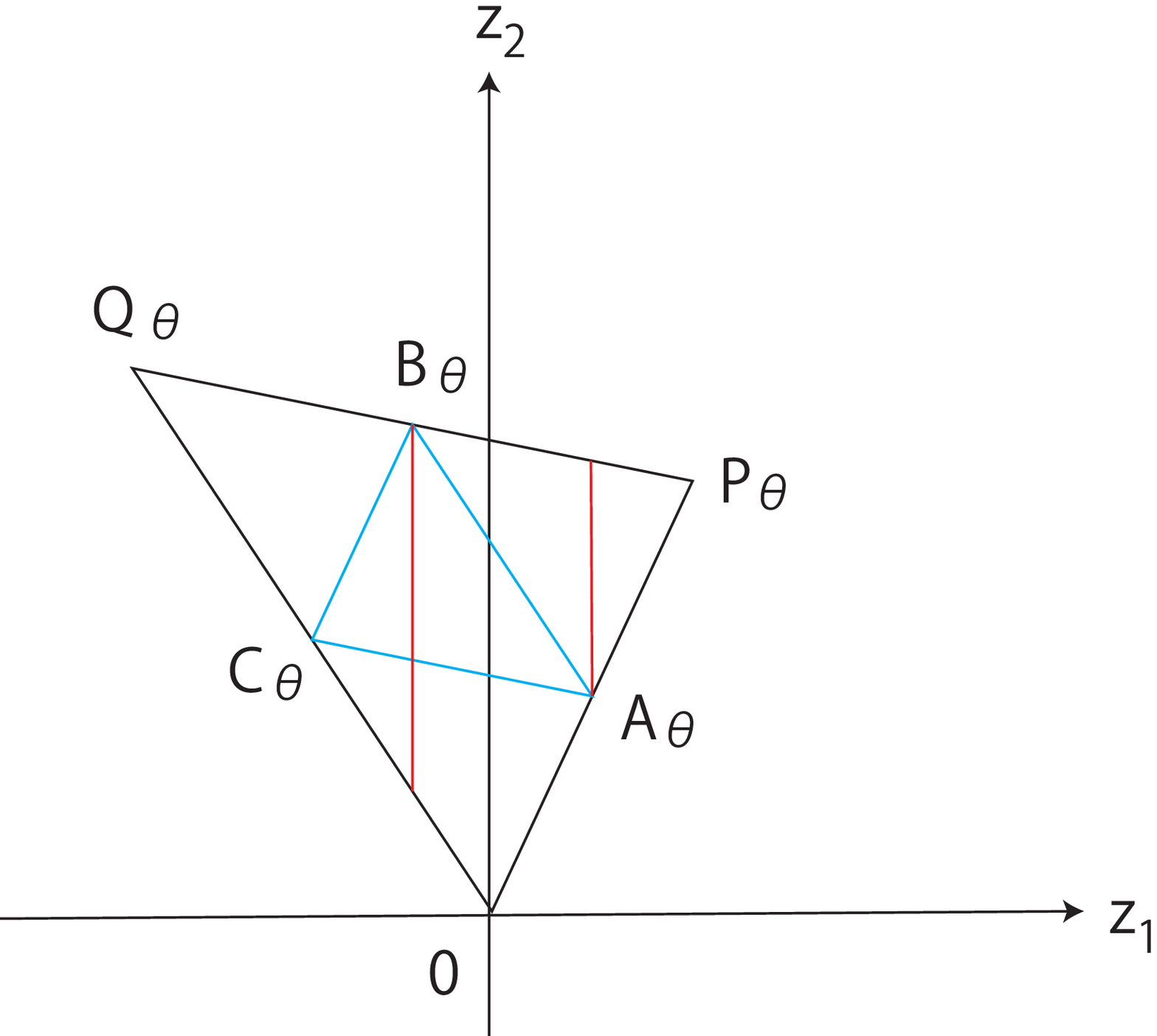}
\caption{Case I.4.1}
\label{triangleI41}
\end{minipage}
\hspace{0.01\linewidth}
\begin{minipage}[h]{0.315\linewidth}
\includegraphics[width=\linewidth]{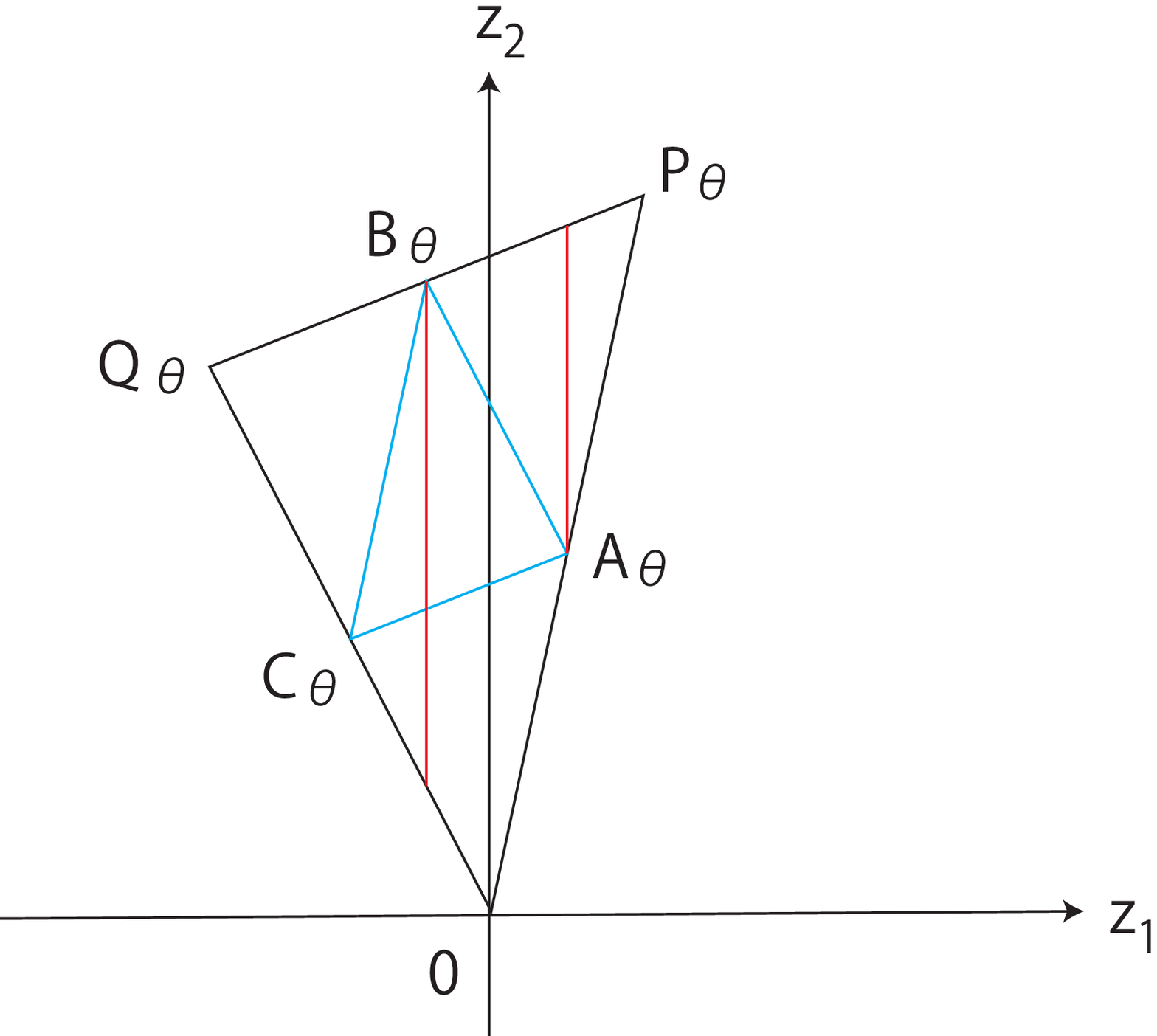}
\caption{Case I.4.2}
\label{triangleI42}
\end{minipage}
\hspace{0.01\linewidth}
\begin{minipage}[h]{0.315\linewidth}
\includegraphics[width=\linewidth]{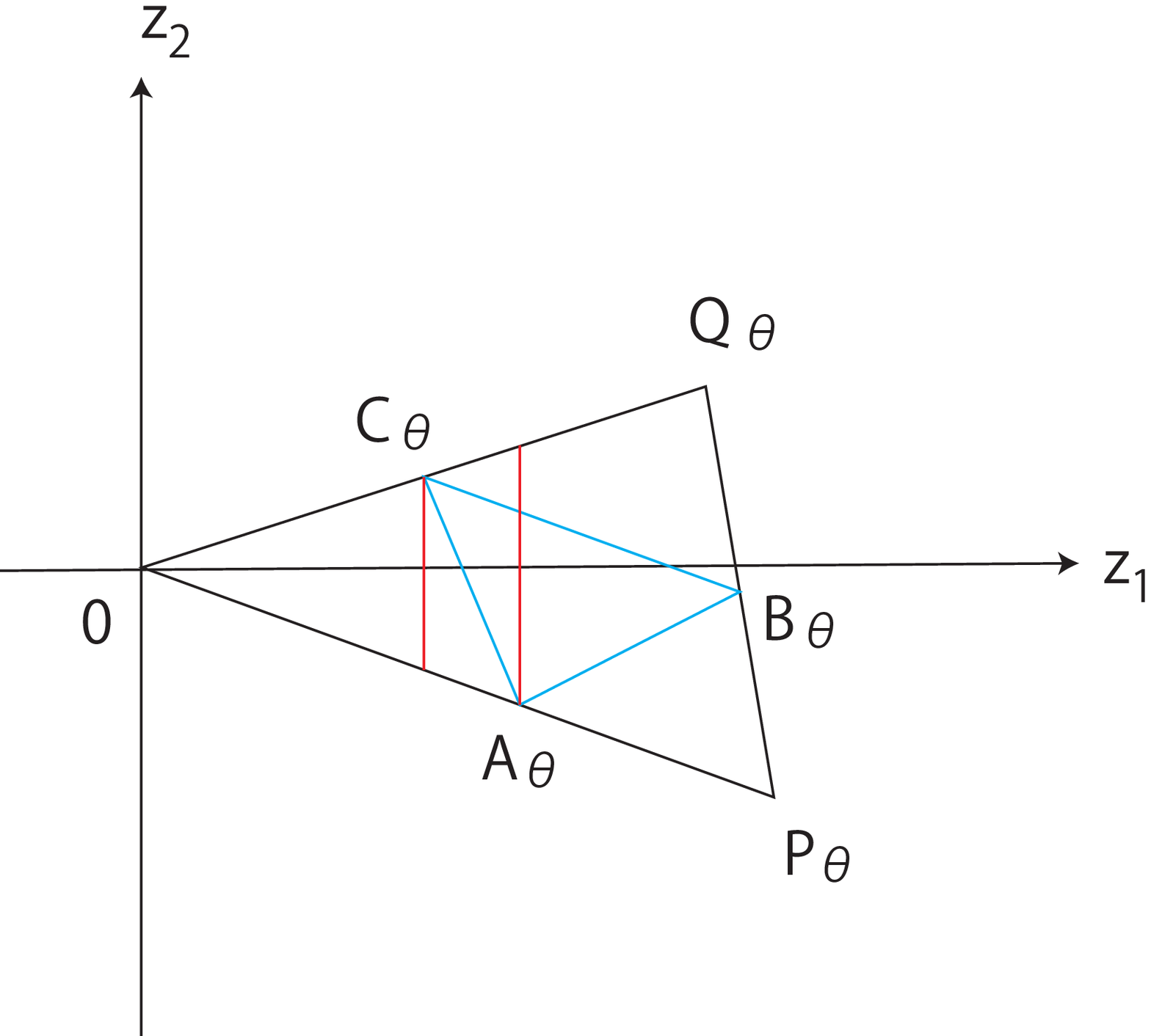}
\caption{Case II.1}
\label{triangleII1}
\end{minipage}
\hspace{0.01\linewidth}
\begin{minipage}[h]{0.315\linewidth}
\includegraphics[width=\linewidth]{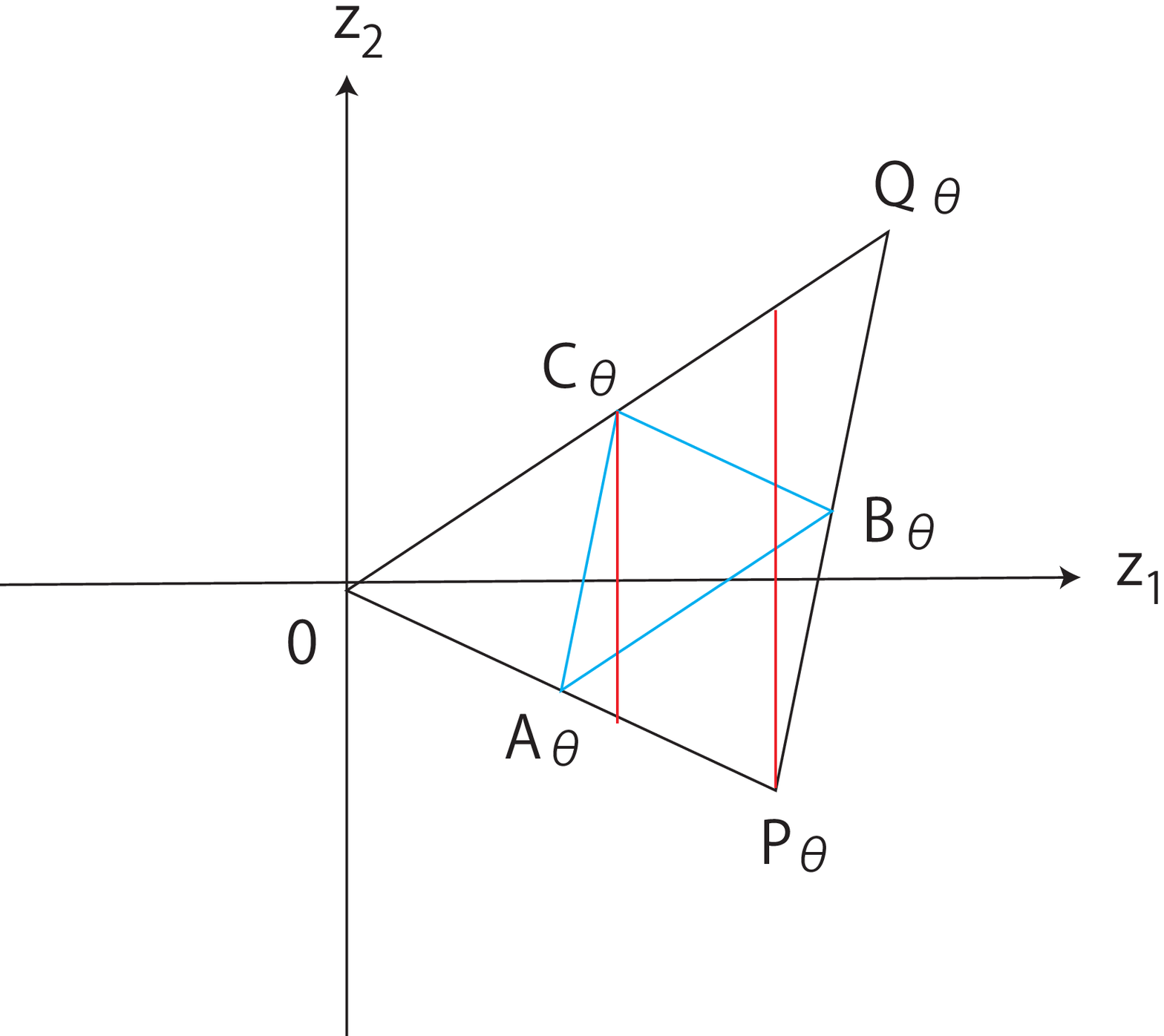}
\caption{Case II.2.1}
\label{triangleII21}
\end{minipage}
\hspace{0.01\linewidth}
\begin{minipage}[h]{0.315\linewidth}
\includegraphics[width=\linewidth]{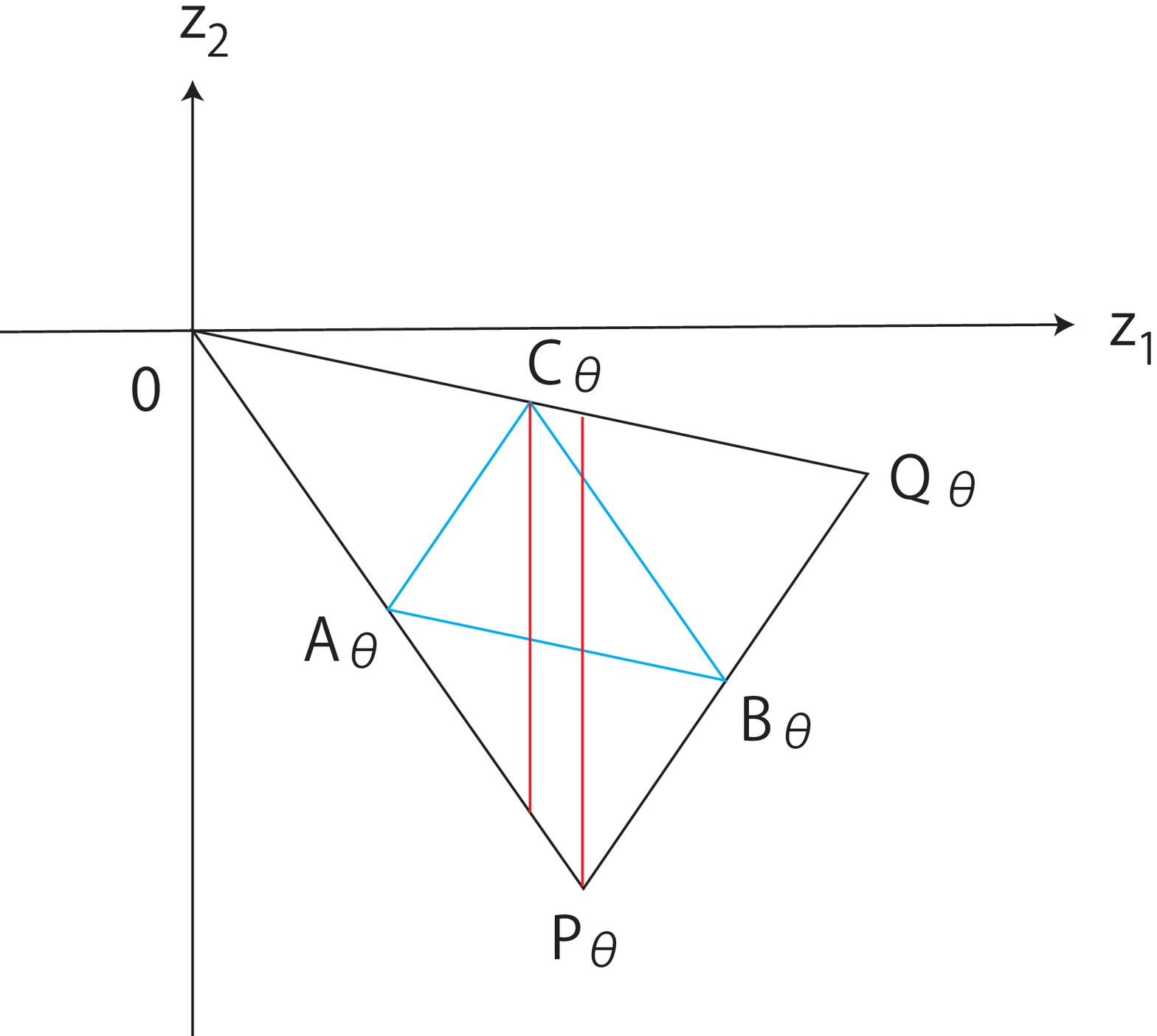}
\caption{Case II.2.2}
\label{triangleII22}
\end{minipage}
\hspace{0.01\linewidth}
\begin{minipage}[h]{0.315\linewidth}
\includegraphics[width=\linewidth]{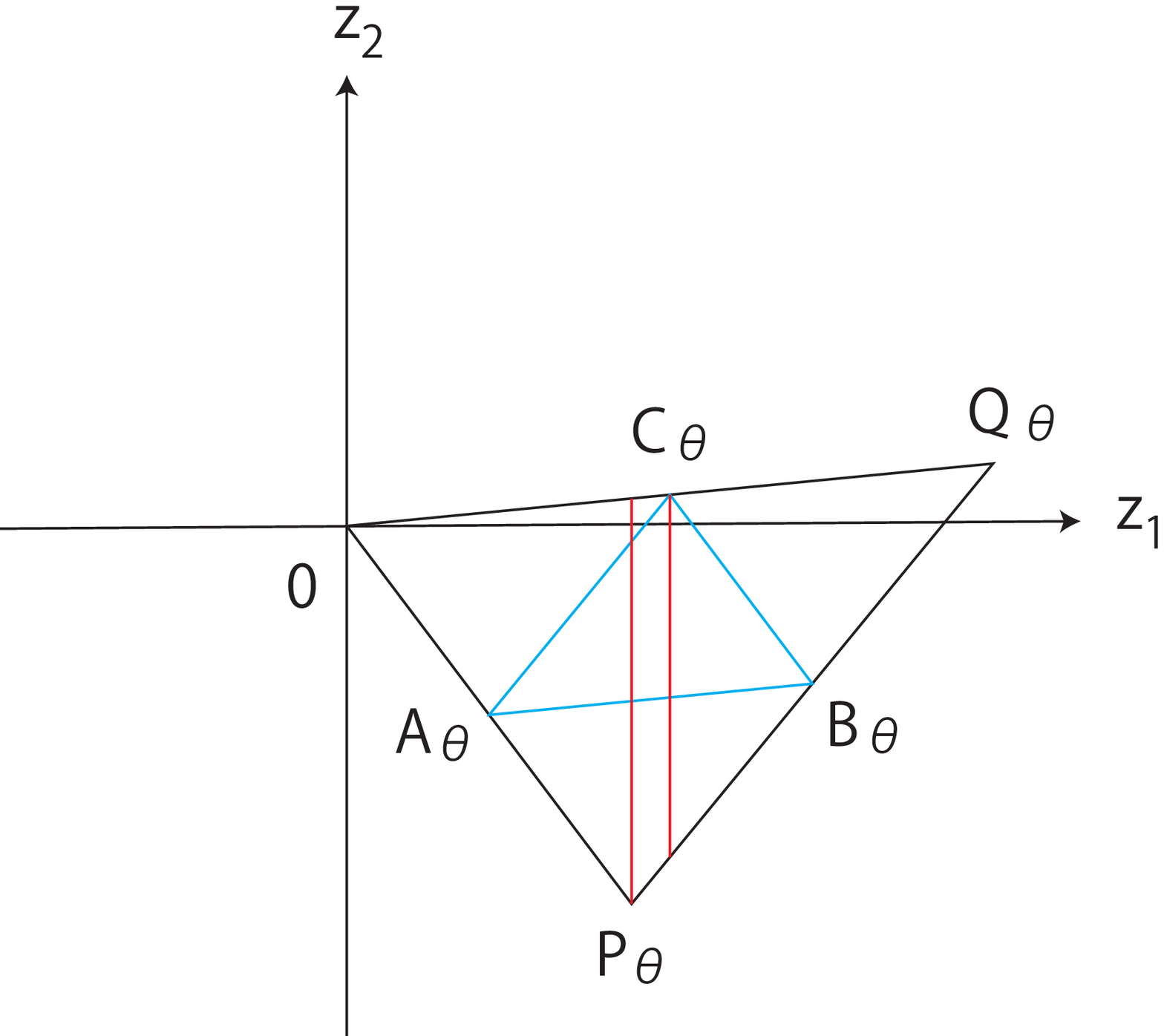}
\caption{Case II.3.1}
\label{triangleII31}
\end{minipage}
\hspace{0.01\linewidth}
\begin{minipage}[h]{0.315\linewidth}
\includegraphics[width=\linewidth]{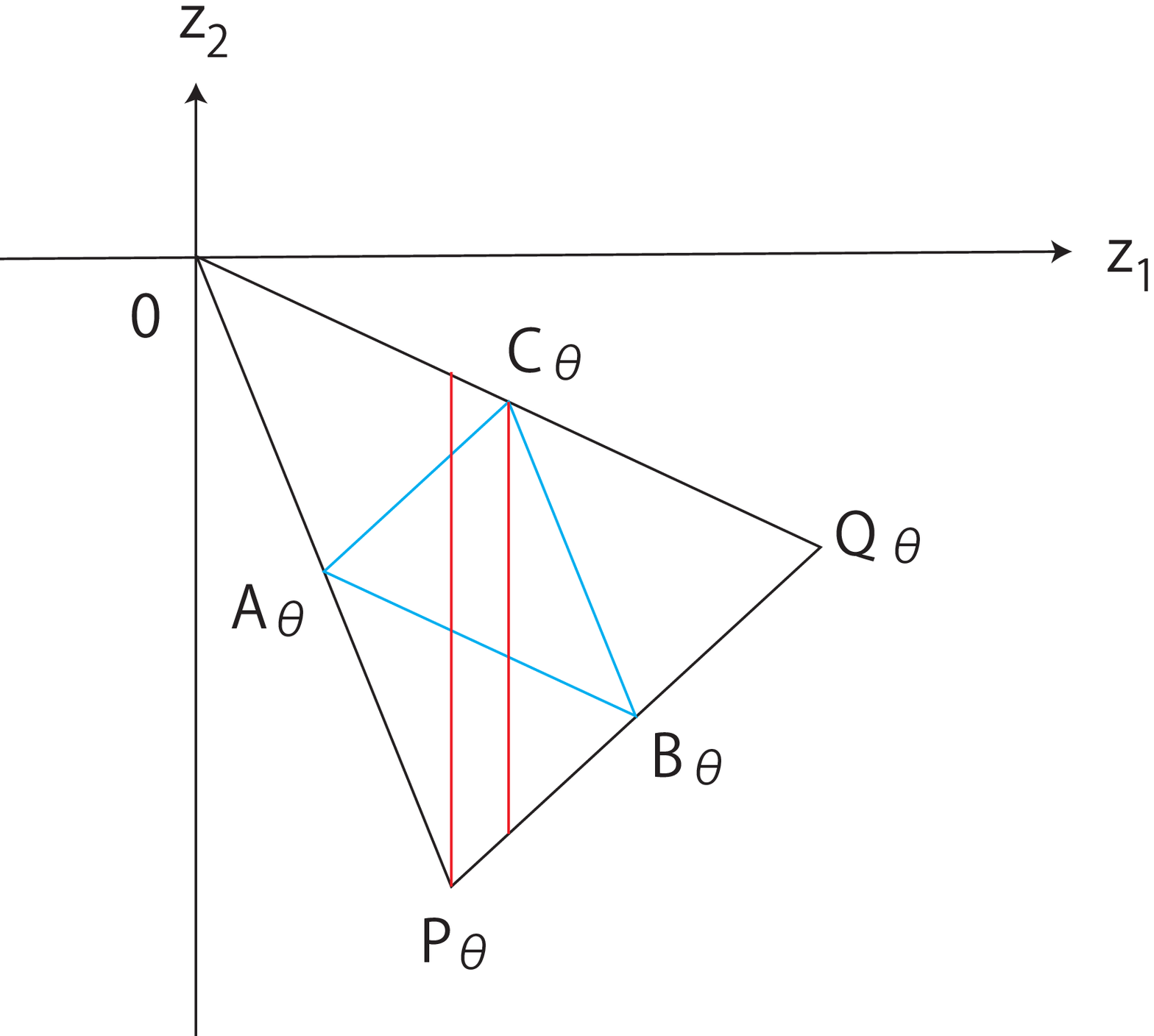}
\caption{Case II.3.2}
\label{triangleII32}
\end{minipage}
\end{figure}

We show the negativity of
\[
\frac{\pd^2 K_{R_\theta \Ome}}{\pd x_1^2}(x)
=-\int_{\pd R_\theta \Ome} \frac{k'(r)}{r}\( x_1 -y_1\) e_1 \cdot n(y) d\sigma (y) 
=-\int_{\pd R_\theta \Ome} \frac{k'(r)}{r} \( x_1 -y_1\) dy_2
\] 
for any $x \in R_\theta (\triangle ABC)$ only in Case I.1. The other cases go parallel. Fix an arbitrary point $x$ in $R_\theta (\triangle ABC)$. 

Suppose $z_1 (C_\theta) \leq x_1 \leq z_1 (A_\theta)$. Then we have the following inequalities in the same argument as in Theorem \ref{uniqueness_revolution} (see Figure \ref{triangleI11}):
\[
\int_{OP_\theta} \frac{k'(r)}{r} \( x_1 -y_1\) dy_2 > 0,\ \int_{Q_\theta O} \frac{k'(r)}{r} \( x_1 -y_1\) dy_2 >0.
\] 
Thus the second derivative of $K_{R_\theta \Ome}$ is negative at such a point $x$.

Suppose $z_1 (A_\theta) \leq x_1 \leq z_1(Q_\theta)$. Put
\begin{align*}
X_\theta&= \( 2x_1- z_1 \( P_\theta \) , \slope \( OP_\theta \) \( 2x_1- z_1 \( P_\theta \) \) \) ,\\
Y_\theta&= \( 2x_1- z_1 \( P_\theta \) , \slope \( OQ_\theta \) \( 2x_1- z_1 \( P_\theta \) \) \) ,\\
Z_\theta&= \( 2x_1 -z_1 \( Q_\theta \), \slope \( OQ_\theta \) \( 2x_1- z_1 \( Q_\theta \) \) \).
\end{align*} 
We have the following inequalities in the same argument as in Theorem \ref{concavity_revolution} (see Figure \ref{triangleI12}):
\[
\int_{X_\theta P_\theta} \frac{k'(r)}{r} \( x_1 -y_1\) dy_2 >0,\
\int_{Q_\theta Z_\theta} \frac{k'(r)}{r} \( x_1 -y_1\) dy_2 >0.
\] 
Let us show the positivity of the contour integral along the line segments $Y_\theta O$ and $OX_\theta$. We remark that, for any $0\leq \de \leq z_1 (X_\theta)$, we have 
\begin{align*}
&\quad \( \slope \( OP_\theta \) + \slope \( OQ_\theta\) \) \( z_1 \( X_\theta \) -\delta \) -2x_2  \\
&\leq \( \slope \( OP_\theta \) + \slope \( OQ_\theta\) \) \( z_1 \( X_\theta \) -\delta \) -2 \( \slope \( OQ_\theta \) \( x_1 - z_1 \( A_\theta \) \) + z_2 \( A_\theta \) \) \\
&=2x_1 \slope \( OP_\theta \) -2z_2 \( P_\theta \) -\delta \( \slope \( OP_\theta \) + \slope \( OQ_\theta \) \) \\
&\leq -\delta \( \slope \( OP_\theta \) + \slope \( OQ_\theta \) \) \\
&\leq 0,
\end{align*} 
where the first and the second inequalities follow from the fact that the point $x$ lie above the line $A_\theta B_\theta$ and that $x_1 \leq z_1 (P_\theta )$, respectively. This inequality implies
\begin{align*}
&\quad \lvert \binom{z_1 \( X_\theta \) -\delta}{\slope \( z_1 \( X_\theta \) -\delta \)} - \binom{x_1}{x_2}\rvert^2 - \lvert \binom{z_1 \( Y_\theta \) -\delta}{\slope \( OQ_\theta \) \( z_1 \( Y_\theta \) -\delta \)} -\binom{x_1}{x_2} \rvert^2 \\
&=\( \( \slope \( OP_\theta\) + \slope \( OQ_\theta\) \) \( z_1 \( OP_\theta \) -\delta \) -2x_2 \) \( \slope \( OP_\theta \) -\slope \( OQ_\theta \) \) \( z_1 \( X_\theta \) -\delta \) \\
&\geq 0
\end{align*} 
for any $0 \leq \de \leq z_1 (X_\theta )$. Hence we obtain 
\begin{equation}\label{XY}
\(\int_{Y_\theta O} + \int_{OX_\theta}\) \frac{k'(r)}{r} \( x_1 -y_1\) dy_2 > 0
\end{equation} 
in the same argument as in Theorem \ref{concavity_revolution} (see also Figure \ref{triangleI12}). Hence the second derivative of $K_{R_\theta \Ome}$ is negative at such a point $x$.

Suppose $z_1 (Q_\theta ) \leq x_1 \leq z_1 (B_\theta )$. In the same argument as in Theorem \ref{concavity_revolution}, we have the following inequalities (see Figure \ref{triangleI13}):
\[
\int_{X_\theta P_\theta} \frac{k'(r)}{r} \( x_1 -y_1\) dy_2 > 0,\ 
\int_{P_\theta Q_\theta} \frac{k'(r)}{r} \( x_1 -y_1\) dy_2 > 0.
\] 
Since the inequality (\ref{XY}) also holds in this case, the second derivative of $K_{R_\theta \Ome}$ is negative at such a point $x$ (see also Figure \ref{triangleI13}).
\end{proof}
\begin{figure}[htbp]
\begin{minipage}[h]{0.315\linewidth}
\includegraphics[width=\linewidth]{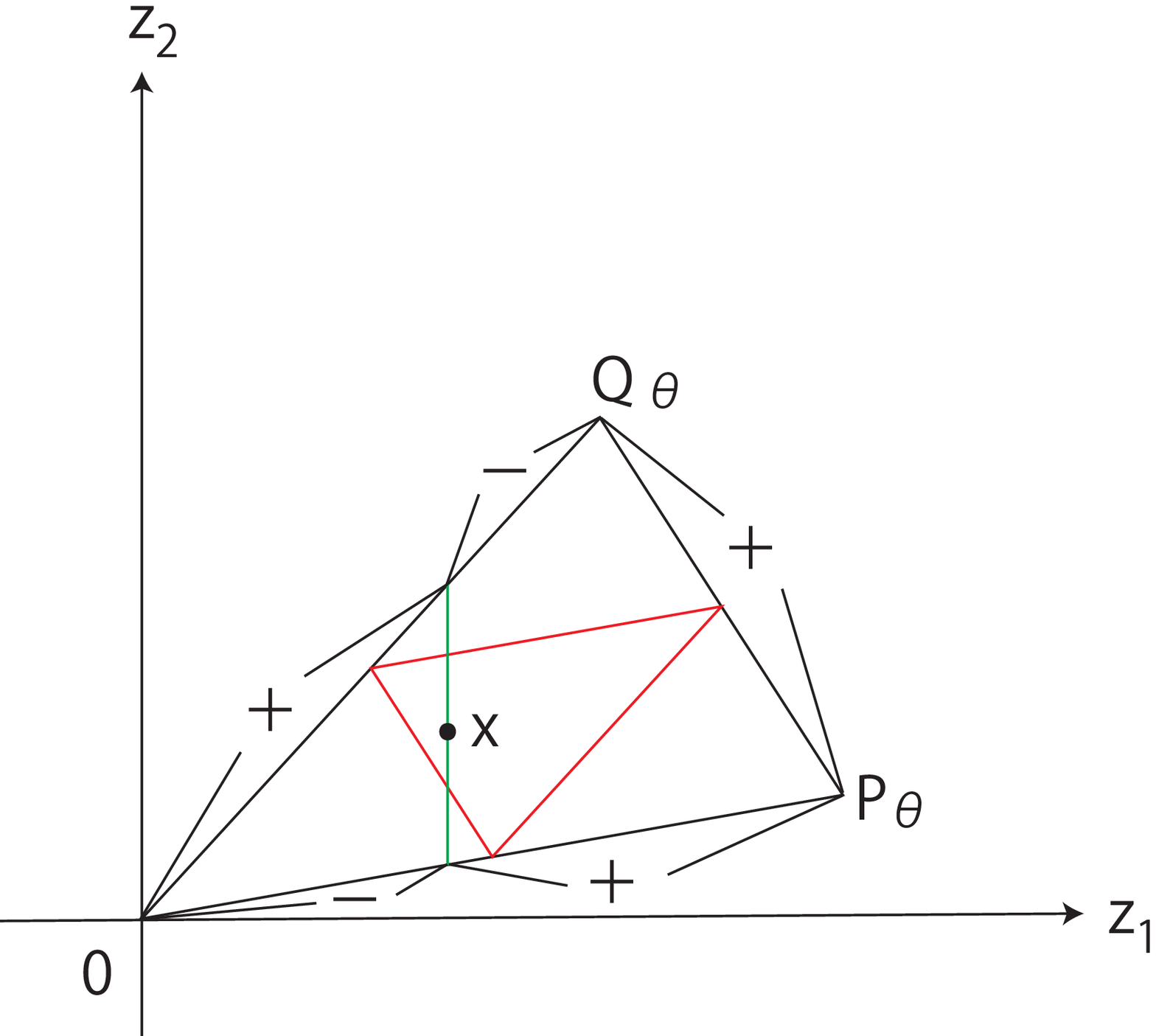}
\caption{}
\label{triangleI11}
\end{minipage}
\hspace{0.01\linewidth}
\begin{minipage}[h]{0.315\linewidth}
\includegraphics[width=\linewidth]{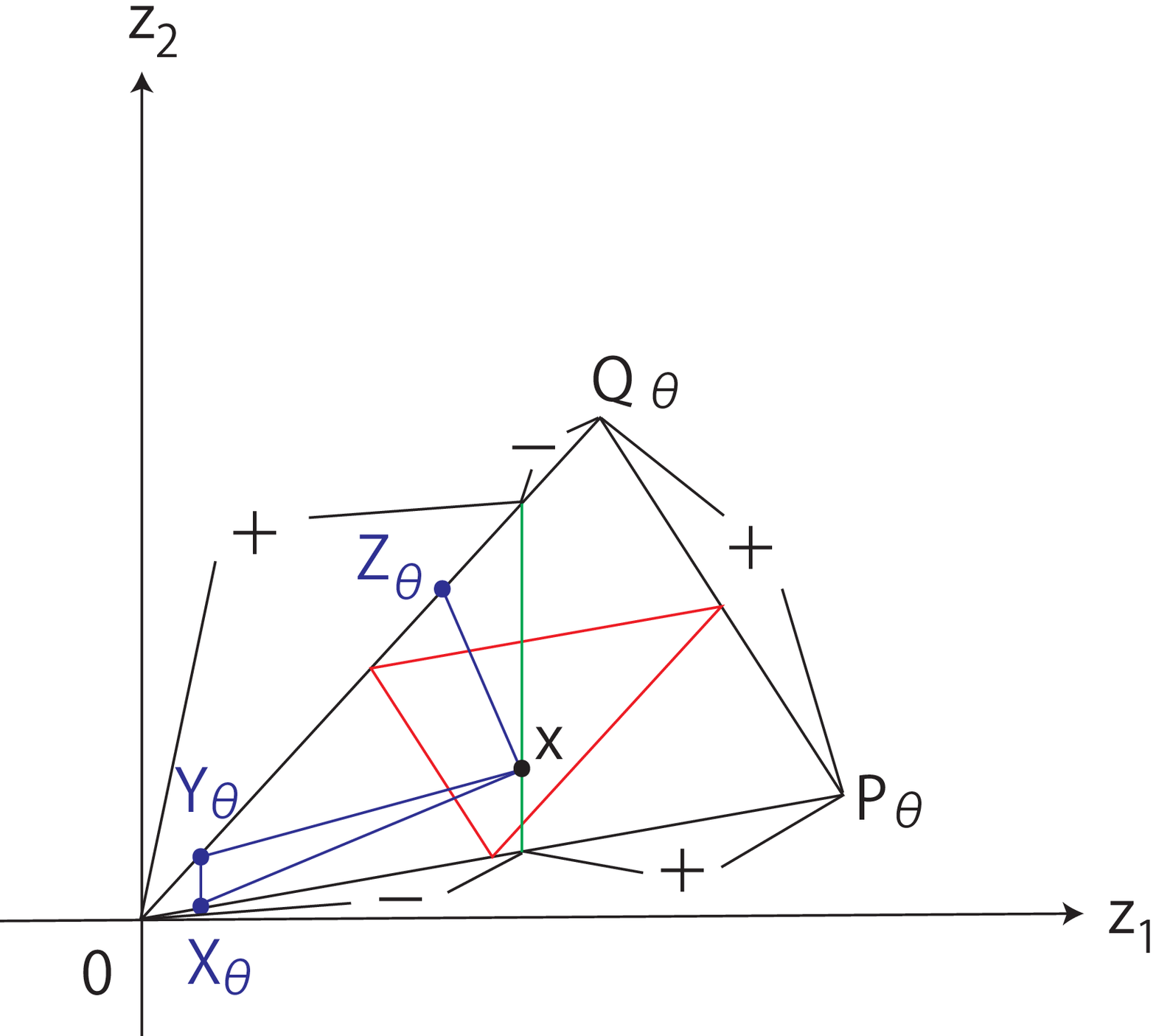}
\caption{}
\label{triangleI12}
\end{minipage}
\hspace{0.01\linewidth}
\begin{minipage}[h]{0.315\linewidth}
\includegraphics[width=\linewidth]{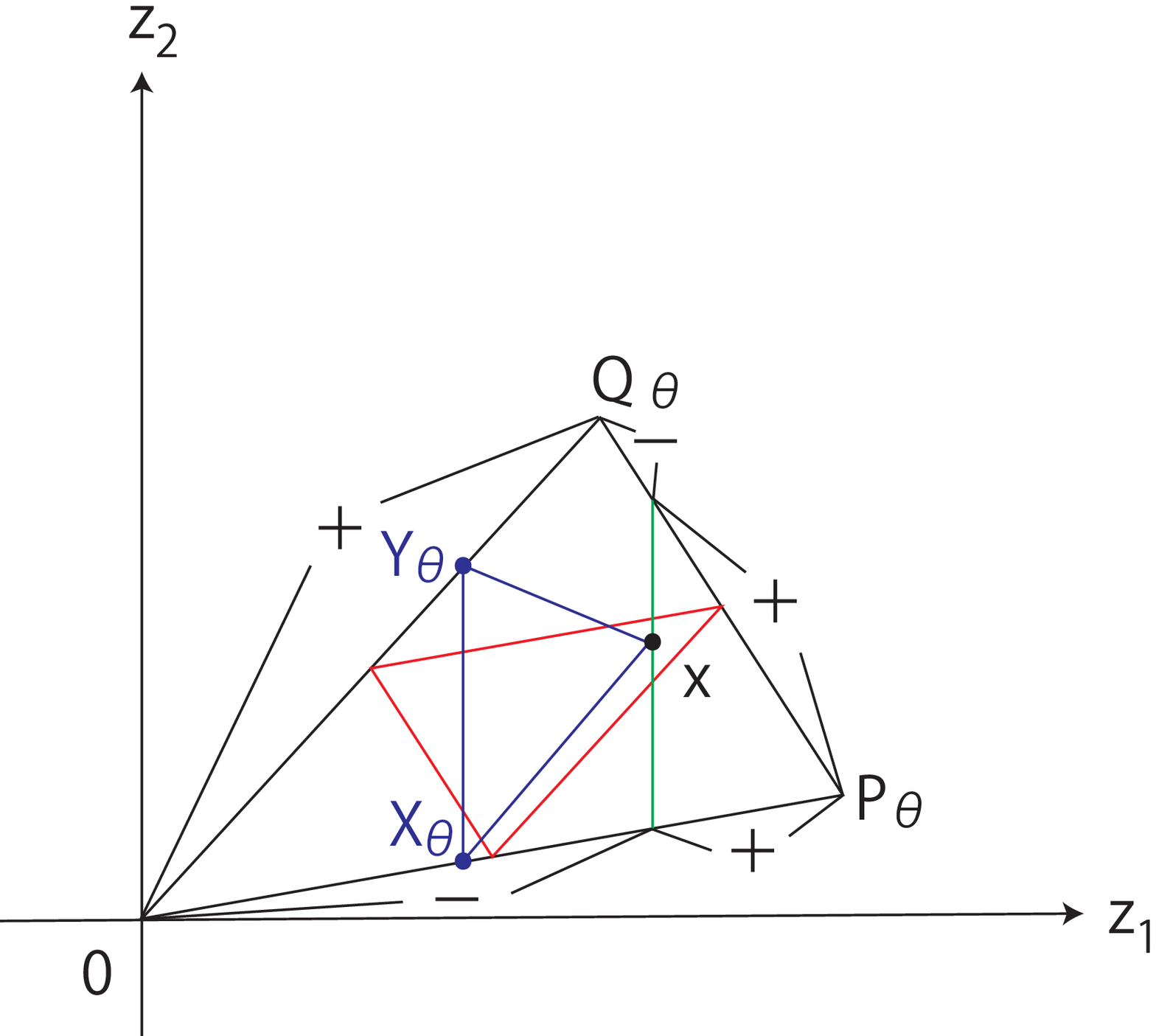}
\caption{}
\label{triangleI13}
\end{minipage}
\end{figure}

\begin{cor}\label{uniqueness_triangle} 
Let $\Ome$ and $k$ be as in Theorem \ref{concavity_triangle}. Then $\Ome$ has a unique $k$-center.
\end{cor}

\begin{rem}
{\rm In the proof of Theorem \ref{concavity_triangle}, we showed the concavity of the potential $K_{\Ome}$ on the triangle $\triangle ABC$. Since the minimal unfolded region is contained in the triangle, we obtained the conclusion.

Unfortunately, this argument does not work for any obtuse triangle (except isosceles triangles). This is because the minimal unfolded region of an obtuse triangle is not contained in the triangle whose vertices are the middle points of the edges (see Example \ref{uf_triangle}).}
\end{rem}
\section{Applications to specific centers}
Let $\Ome$ be a body in $\R^m$. We consider some applications of the results in the previous section. 

Let
\begin{equation}
V_\Ome^{(\al )}(x) =
\begin{cases}
\ds \sign (m-\al ) \int_\Ome r^{\al -m} dy &( 0<\al \neq m ),\\
\ds -\int_\Ome \log r dy & ( \al =m ) ,
\end{cases}
\ x \in \R^m .
\end{equation}
The potential $V_\Ome^{(\al )}$ is called the {\it $r^{\al -m}$-potential} of order $\al$.

\begin{definition}[{\cite[Definition 3.1]{O1}}]
{\rm A point $x$ is called an {\it $r^{\al -m}$-center} of $\Ome$ if it gives the maximum value of $V_\Ome^{(\al )}$.}
\end{definition} 

\begin{thm}[{\cite[Theorem 3.1]{M1}}]
If $0<\al \leq 1$, and if $\Ome$ is convex, then $\Ome$ has a unique $r^{\al -m}$-center.
\end{thm}

\begin{thm}[{\cite[Theorem 3.15]{O1}}]
If $\al \geq m+1$, then $\Ome$ has a unique $r^{\al -m}$-center.
\end{thm}

\begin{thm}[{\cite[Theorem 3.8]{O3}}]
Let $\tilde{\Ome}$ be a compact convex set in $\R^m$, and
\begin{align*}
f(\al ) 
&= \frac{\sqrt{m+1-\al}}{2} \( 2+\frac{3}{\( 1+ \( 4 \( 4 \sqrt{\frac{m+2-\al}{m+1-\al}} +\frac{1}{2} \sqrt{\frac{m+1-\al}{m+2-\al}} \)^2 +1 \)^{-(m+2-\al )/2} \)^{1/(m-2)} -1} \) \\
&\quad \times \( 4 \sqrt{\frac{m+2-\al}{m+1-\al}} +\frac{1}{2} \sqrt{\frac{m+1-\al}{m+2-\al}} \) -1 ,\ 1< \al <m+1 .
\end{align*}
For any $1<\al <m+1$, if $\de \geq f(\al )\diam \tilde{\Ome}$, then the parallel body $\tilde{\Ome}+\de B^m = \left\{ \tilde{y} + \de w \lvert \tilde{y} \in \tilde{\Ome},\ w \in B^m \right\} \right.$ has a unique $r^{\al -m}$-center.
\end{thm}

\begin{prop}\label{uniquenessV}
Let $\Ome$ be as in Theorem \ref{concavity_revolution} or \ref{concavity_triangle}. For any $1< \al <m+2$, $\Ome$ has a unique $r^{\al -m}$-center.
\end{prop}

\begin{proof}
If $1<\al <m+2$, direct computation shows that the kernel of $V_\Ome^{(\al )}$ satisfies the assumption as in Theorem \ref{concavity_revolution} or \ref{concavity_triangle}.
\end{proof}

\begin{rem}
{\rm Let us remark the value of Proposition \ref{uniquenessV}. We newly proved the uniqueness of an $r^{\al -m}$-center for $1<\al <m+1$ when $\Ome$ cannot be obtained as any parallel body like as Example \ref{ex_triangle} or Example \ref{ex_arcon}.}
\end{rem}

Let 
\begin{equation}
A_\Ome (x,h) = \int_\Ome \frac{h}{\( r^2+h^2 \)^{(m+1)/2}} dy ,\ x \in \R^m ,\ h>0.
\end{equation} 
It is well-known that the function $A_\Ome$ satisfies the Laplace equation for the upper half space
\begin{equation}
\De A_\Ome (x,h) = \( \sum_{j=1}^{m} \frac{\pd^2}{\pd x_j^2} + \frac{\pd^2}{\pd h^2} \) A_\Ome (x,h) =0,\ x \in \R^m ,\ h>0,
\end{equation}
and the boundary condition
\begin{equation}
\lim_{h\to 0^+} A_\Ome (x,h) =\frac{\sigma_m \( S^m\)}{2} \chi_\Ome (x) ,\ x \in \R^m \sm \pd \Ome .
\end{equation}
The function $A_\Ome (x,h)$ has a geometric meaning as below: Let $x \in \R^m$ and $h>0$; Define the map $p_{(x,h)} :\Ome \to S^m$ by
\begin{equation}
p_{(x,h)} (y) = \frac{(y,0)-(x,h)}{\lvert (y,0)-(x,h) \rvert} = \frac{(y,0)-(x,h)}{\sqrt{r^2+h^2}};
\end{equation}
The {\it solid angle} of $\Ome$ at $(x,h)$ is defined as the spherical Lebesgue measure of the image $p_{(x,h)} (\Ome )$; Direct calculation shows that $A_\Ome(x,h)$ coincides with the solid angle of $\Ome$ at $(x,h)$. In other words, the function $A_\Ome (x,h)$ gives the ``visibility'' of $\Ome$ at the point $(x,h)$.

On the other hand, in \cite{Sh}, the function $A_\Ome (x,h)$ was introduced by Katsuyuki Shibata to give an answer for PISA's problem ``Where should we put a streetlight in a triangular park?''. Shibata called a maximizer of $A_\Ome (\cdot ,h)$ an {\it illuminating center} of $\Ome$ of height $h$.

\begin{thm}[{\cite[Theorem 5.32, Proposition 5.33, Theorem 5.36]{Sak}}]
\begin{enumerate}
\item[$(1)$] If $h \geq \sqrt{m+2} \tilde{D}(\Ome )$, where $\tilde{D}(\Ome )$ is a slight improvement of $D(\Ome )$, then $\Ome$ has a unique illuminating center.
\item[$(2)$] If $h \leq \sqrt{2/(m-1)} d(\Ome )$, if $\Ome$ is convex, and if $Uf(\Ome )$ is contained in the interior of $\Ome$, then $\Ome$ has a unique illuminating center.
\item[$(3)$] Let $\tilde{\Ome}$ be a compact convex set in $\R^m$. If $\de \geq \sqrt{(m+2)(m-1)/2} \diam \tilde{\Ome}$, then, for any $h$, the parallel body $\tilde{\Ome} + \de B^m$ has a unique illuminating center.
\end{enumerate} 
\end{thm}

\begin{prop}\label{uniquenessA}
Let $\Ome$ be as in Theorem \ref{concavity_revolution} or \ref{concavity_triangle}. For any $h>0$, $\Ome$ has a unique illuminating center.
\end{prop}

\begin{proof}
Direct computation shows that the kernel of $A_\Ome$ satisfies the assumption as in Theorem \ref{concavity_revolution} or \ref{concavity_triangle} for any $h$.
\end{proof}

\begin{rem}
{\rm Let us remark the value of Proposition \ref{uniquenessA}. We newly proved the uniqueness of an illuminating center without the assumption of $h$ when $\Ome$ cannot be obtained as any parallel body like as Example \ref{ex_triangle} or Example \ref{ex_arcon}.}
\end{rem}


\no 
Faculty of Education and Culture,\\
University of Miyazaki,\\
1-1, Gakuen Kibanadai West, Miyazaki city, Miyazaki prefecture, 889-2155, Japan\\
E-mail: sakata@cc.miyazaki-u.ac.jp
\end{document}